\documentclass[12pt]{article}

\usepackage{amssymb}
\usepackage{amsthm}
\usepackage{amsmath}
\usepackage{graphicx}
\usepackage{fullpage}
\usepackage{color}
 \numberwithin{equation}{section}
 \usepackage{mathtools}
 \mathtoolsset{showonlyrefs}
\allowdisplaybreaks
\usepackage[pdftex]{hyperref}

\theoremstyle{plain}
\newtheorem{thm}{Theorem}[section]
\newtheorem{cor}[thm]{Corollary}

\newtheorem{lem}[thm]{Lemma}
\newtheorem{prop}[thm]{Proposition}

\newtheorem{innercustomthm}{Theorem}
\newenvironment{customthm}[1]
  {\renewcommand\theinnercustomthm{#1}\innercustomthm}
  {\endinnercustomthm}
  
\theoremstyle{definition}
\newtheorem{defn}[thm]{Definition}

\theoremstyle{remark}

\newtheorem{rem}[thm]{Remark}

\newcommand{\M}{\mathbb{M}}
\newcommand{\N}{\mathbb{N}}
\newcommand{\R}{\mathbb{R}}

\newcommand{\ffint}{\iint_{Q_r} \!\!\!\!\!\!\!\!\!\!\!\!\!\!\text{-----}}
\newcommand{\ffintXT}{\iint_{Q_r(x,t)} \!\!\!\!\!\!\!\!\!\!\!\!\!\!\!\!\!\!\!\!\!\! \text{-----}}
\newcommand{\ffRint}{\iint_{Q_R} \!\!\!\!\!\!\!\!\!\!\!\!\!\!\text{-----}}

\newcommand{\intQrtau}{\iint_{Q_{r\tau}} \!\!\!\!\!\!\!\!\!\!\!\!\!\!\!\!\text{-----}}
\newcommand{\intQvarkr}{\iint_{Q_{\vartheta^kr}} \!\!\!\!\!\!\!\!\!\!\!\!\!\!\!\!\!\!\text{-----}}
\newcommand{\fffint}{\iint_{Q_1} \!\!\!\!\!\!\!\!\!\!\!\!\!\!\text{-----}}
\newcommand{\fthetaint}{\iint_{Q_{\vartheta}} \!\!\!\!\!\!\!\!\!\!\!\!\!\!\text{-----}}

\newcommand{\Div}{\text{div}}
\newcommand{\bp}{\begin{proof}[\ensuremath{\mathbf{Proof}}]}
\newcommand{\ep}{\end{proof}}
\newcommand{\bv}{{\bf v}}

\newcommand{\bu}{{\bf u}}
\newcommand{\bg}{{\bf g}}

\newcommand{\bw}{{\bf w}}

\newcommand{\Mmn}{\M^{m\times n}}
\newcommand{\Mnn}{\M^{n\times n}}

\begin{document}

\title{Partial regularity for doubly nonlinear parabolic systems of the first type}

\author{Ryan Hynd\footnote{Department of Mathematics, MIT.  Partially supported by NSF grant DMS-1554130 and an MLK visiting professorship.}}

\maketitle

\begin{abstract}
We study solutions $\bv$ of the parabolic system of PDE 
$$
\partial_t\left(D\psi(\bv)\right)=\Div DF(D\bv).
$$
Here $\psi$ and $F$ are convex functions, and this is a model equation for more general doubly nonlinear evolutions that arise in the study of phase transitions in materials.  We show that if $\bv$ is an appropriately defined weak solution, then $D\bv$ is locally H\"older continuous except for possibly on a lower dimensional subset of the domain of $\bv$.  Our proof is based on compactness properties of solutions,  two  integral identities and a fractional time derivative estimate for $D\bv$. 
\end{abstract}

\section{Introduction}
A doubly nonlinear evolution is a flow that may involve a time derivative of a nonlinear function of a quantity of interest. These types of flows arise in various physical models for phase transitions as detailed in the monograph \cite{Vis}. Notable examples include the Stefan problem, which concerns a classical model of phase transition in solid-liquid systems \cite{Friedman, Friedman2, Oleinik, Rubin, Stefan}; the Hele-Shaw and Muskat problems, which involve the dynamics of two immiscible viscous fluids \cite{Constantin, Richardson, Richardson2, Saff}; the flow of an incompressible fluid through a porous medium \cite{AltLuck, DiaThe}; and the study of interfacial (Gibbs--Thomson) effects that occur during phase nucleation, growth and coarsening \cite{Chalmers, Gurtin, Gurtin2, Perez, Woodruff}. In this paper, we will consider doubly nonlinear evolutions that are also parabolic systems. 

\par  In what follows, we will focus on solutions $\bv: U\times(0,T)\rightarrow \R^m$ of PDE systems of the form
\begin{equation}\label{mainPDE}
\partial_t\left(D\psi(\bv)\right)=\Div DF(D\bv).
\end{equation} 
Here $U\subset\R^n$ is a bounded domain with smooth boundary, $T>0$, and $\psi: \R^m\rightarrow \R$ and $F: \M^{m\times n}\rightarrow \R$ are convex. We will denote $\M^{m\times n}$ as the space of $m\times n$ matrices with real entries and $\bv=(v^1,\dots, v^m)$ for the $m$ component functions $v^i=v^i(x,t)$ of $\bv$. Also note
$$
\bv_t=(v^1_t, \dots, v^m_t)\in \R^m\quad \text{and}\quad D\bv
=\left(\begin{array}{ccc}
v^1_{x_1} & \dots  & v^1_{x_n}\\
& \ddots &  \\
v^m_{x_1} & \dots  & v^m_{x_n}\\
\end{array}\right)\in \M^{m\times n}
$$
are the respective time derivative and spatial gradient matrix of $\bv$.

\par  Writing $\psi(w)=\psi(w^1,\dots, w^m)$ for $w=(w^i)\in \R^m$ and $F(M)=F(M^1_1, \dots, M^m_n)$ for $M=(M^i_j)\in \Mmn$, the system \eqref{mainPDE} can also be posed as the system of $m$ equations
$$
\partial_t\left(\psi_{w_i}(\bv)\right)=\sum^n_{j=1}\left(F_{M^i_j}(D\bv)\right)_{x_j}, \quad i=1, \dots, m. 
$$
Carrying out the derivatives, we may also write
$$
\sum^n_{j=1}\psi_{w_iw_j}(\bv)v^j_t=\sum^m_{k=1}\sum^n_{j,\ell=1}F_{M^i_jM^k_\ell}(D\bv) v^k_{x_\ell x_j}, \quad i=1, \dots, m. 
$$
In particular, we see that the system \eqref{mainPDE} is a type of quasilinear system of parabolic PDE.

\par The principle assumptions that we will make in this work are that there are constants $\theta,\lambda, \Theta,\Lambda>0$ such that
\begin{equation}\label{UnifConv}
\theta |w_1-w_2|^2\le \left(D\psi(w_1)-D\psi(w_2)\right)\cdot\left(w_1-w_2\right)\le 
\Theta |w_1-w_2|^2\end{equation}
for each  $w_1,w_2\in \R^m$ and
\begin{equation}\label{UnifConv2}
\lambda |M_1-M_2|^2\le \left(DF(M_1)-DF(M_2)\right)\cdot\left(M_1-M_2\right)\le \Lambda |M_1-M_2|^2
\end{equation}
for each $M_1,M_2\in \Mmn$.  Here we are using the notation $M\cdot N:=\text{tr}(M^tN)$ and $|M|:=(M\cdot M)^{1/2}$ for each $M,N\in\Mmn$. In other words, $\psi$ and $F$ will always assumed to be (at least) continuously differentiable, uniformly convex and to grow quadratically.  We also point out that these assumptions require $D\psi$ and $DF$ to be globally Lipschitz mappings which are therefore differentiable almost everywhere on $\R^m$ and $\Mmn$, respectively.  

\par Our central results involve the regularity of weak solutions whose definition is postponed until later in this work (Definition \ref{weakSolnLoc}).  We remark that our definition of weak solution requires more integrability than what is typically asked of a mapping $\bv:U\times(0,T)\rightarrow\R^m$ to be a solution of \eqref{mainPDE} (such as in \cite{AltLuck, DiaThe, Vis}).  Nevertheless, we show in the appendix how to construct weak solutions for given initial and boundary conditions with a natural implicit time scheme.  So while our result may not apply to the widest class of solutions of \eqref{mainPDE}, it holds for a class of solutions that exist and arise in a most natural way. 

\begin{customthm}{1}\label{mainThm}
Assume $\psi\in C^2(\R^m)$ and $F\in C^2(\M^{m\times n})$ and that \eqref{UnifConv} and \eqref{UnifConv2} hold. Suppose $\bv$ is a weak solution of \eqref{mainPDE}. Then there is an open subset ${\cal O}\subset U\times(0,T)$ whose complement has Lebesgue measure 0 for which $D\bv$ is H\"older continuous in a neighborhood of each point in ${\cal O}$. 
\end{customthm}
\par Our approach to proving Theorem \ref{mainThm} is as follows. First, we will derive two integral identities for weak solutions which provides us with some local energy estimates. Next, we will establish that every appropriately bounded sequence of solutions of the system \eqref{mainPDE} has a subsequence that converges strongly to another weak solution \eqref{mainPDE}.  Then we will exploit this compactness to argue that certain integral quantities decay as they would for solutions of the linearization of \eqref{mainPDE}. Finally, we will use versions of Lebesgue's differentiation theorem and Campanato's criterion with parabolic cylinders (instead of Euclidean balls) to obtain partial H\"older continuity of the gradients of weak solutions.

\par  It turns out that it is also possible to establish a certain fractional time differentiability of $D\bv$. This property can be used to demonstrate that the set ${\cal O}$ in the statement of Theorem \ref{mainThm} can be selected to be lower dimensional.  We state the following refinement of Theorem \ref{mainThm} postponing the definition of Parabolic Hausdorff measure ${\cal P}^s$ $(0\le s\le n+2)$ until the final section of this paper (Definition \ref{ParaHausMeas}). We only note here that the Lebesgue outer measure on $\R^{n+1}$ is absolutely continuous with respect to ${\cal P}^{n+2}$, which means that the following assertion improves Theorem \ref{mainThm}.
\begin{customthm}{2}\label{secondThm}
Assume $\psi\in C^2(\R^m)$ and $F\in C^2(\M^{m\times n})$ and that \eqref{UnifConv} and \eqref{UnifConv2} hold. Suppose $\bv$ is a weak solution of \eqref{mainPDE} and  
$$
{\cal O}=\{(x,t)\in U\times (0,T): D\bv\; \text{is H\"older continuous in some neighborhood of $(x,t)$}\}.
$$
Then there is $\beta\in (0,1)$ for which 
$$
{\cal P}^{n+2-2\beta}(U\times (0,T)\setminus {\cal O})=0.
$$
\end{customthm}

\par We note that partial regularity statements such as what is claimed above are quite natural to consider for systems. Even in the stationary case of  \eqref{mainPDE}, weak solutions $\bu:U\rightarrow \R^m$ of
$$
-\Div(DF(D\bu))=0
$$
may fail to be regular at some points \cite{DeG, Guisti, Lawson}.  We also note that Theorems \ref{mainThm} and \ref{secondThm} have already been established for the gradient flow system
$$
\partial_t\bv=\Div DF(D\bv),
$$
which corresponds to \eqref{mainPDE} when $\psi(w)=\frac{1}{2}|w|^2$ \cite{Campanato1}. In more recent work \cite{DuzMin, DuzMinSte}, the stronger assertion ${\cal P}^{n-\delta}(U\times (0,T)\setminus {\cal O})=0$ for some $\delta>0$ was verified.
It is also worth mentioning that partial H\"older continuity has been established for weak solutions when restricted to the parabolic boundary of $U\times (0,T)$ \cite{BDMI, BDMII}.

\par    Observe that if we set $\bw=D\psi(\bv)$, $\bw$ formally satisfies the system
\begin{equation}\label{MingionePDE}
\partial_t\bw=\Div A(D\bw,\bw).
\end{equation} 
Here
$$
A(\zeta,w):=DF(D^2\psi^*(w)\zeta)
$$
for $\zeta\in \Mmn$ and $w\in \R^m$; $\psi^*$ is the Legendre transform of $\psi$.  Using the results of F. Duzaar and G. Mingione \cite{DuzMin}, we would be able to conclude that
$D\bw$ is partially H\"older continuous if \eqref{MingionePDE} is uniformly parabolic. That is, if
\begin{equation}\label{SpecA}
D_\zeta A(\zeta, w)\xi\cdot \xi = D^2F(D^2\psi^*(w)\zeta)D^2\psi^*(z)\xi\cdot \xi\ge  \varepsilon |\xi|^2
\end{equation}
for all $\zeta,\xi\in \Mmn$, $w\in \R^m$, and for some $\varepsilon>0$. Since $D\bv=D^2\psi^*(\bw)D\bw$, we would essentially be in position to deduce 
Theorem \ref{secondThm} provided that $D^2\psi^*$ is H\"older continuous.

\par It is possible to find $\psi$ and $F$ for which \eqref{SpecA} holds. This is the case, for example, when \eqref{SpecA} is a scalar equation $(m=1)$. This is also the case when $\psi(w)=\frac{1}{2}|w|^2$ or if $F(M)=\frac{1}{2}|M|^2$, which was studied in detail in \cite{Max}. However, 
\eqref{SpecA} does not hold for every $\psi\in C^2(\R^m)$ and $F\in C^2(\Mmn)$ satisfying \eqref{UnifConv} and \eqref{UnifConv2}, respectively. A simple class of examples for which \eqref{SpecA} fails can be found for $n=1$ so that ${\mathbb M}_{m1}$ can be identified with $\R^m$.  For any two symmetric $m\times m$ positive definite matrices $B$ and $C$ such that $CB^{-1}$ is not positive definite, \eqref{SpecA} does not hold for 
$$
\psi(w)=\frac{1}{2}Bw\cdot w\quad \text{and} \quad F(M)=\frac{1}{2}CM\cdot M
$$
($w, M\in \R^m$). Consequently, it is worthwhile to study the doubly nonlinear parabolic system \eqref{mainPDE} as it is written and not always change variables to put it in the form \eqref{MingionePDE}.

\par We also note that we only consider doubly nonlinear evolutions that are uniformly parabolic, which is a subset of the class of doubly nonlinear evolutions that are degenerate parabolic 
systems. We anticipate 
that our arguments can be extended to address the regularity of solutions to degenerate systems as has been done for doubly nonlinear parabolic equations. In our context, a doubly nonlinear 
parabolic equation corresponds to $m=1$ so that the system \eqref{mainPDE} reduces to a single PDE
$$
\partial_t(\psi'(v))=\Div(DF(Dv))
$$
for a scalar function $v:U\times(0,T)\rightarrow \R$. For many examples of $\psi$ and $F$ that satisfy suitable regularity, convexity and compatible growth conditions, there is a local H\"older estimate for solutions \cite{Ivanov, Porzio, Vespri}. 
For example, H\"older regularity has been shown in various contexts for the particular case 
\begin{equation}\label{pTrudinger}
\partial_t(|v|^{p-2}v)=\Div(|Dv|^{p-2}Dv).
\end{equation}
\cite{Kuusi, Tru}. It would be very interesting to pursue the regularity of systems which have growth and convexity properties analogous to \eqref{pTrudinger}.

\par Finally, we remark that equation \eqref{mainPDE} is known as a doubly nonlinear parabolic system of the {\it first} type.  A doubly nonlinear parabolic system of the {\it second} type is of the form
\begin{equation}\label{SecondTypePDE}
D\psi\left(\bv_t\right)=\Div DF(D\bv).
\end{equation} 
We believe this terminology is due to A. Visintin \cite{Vis}.  The main difference between \eqref{mainPDE} and \eqref{SecondTypePDE} is that \eqref{mainPDE} is quasilinear while \eqref{SecondTypePDE} is fully nonlinear.  Nevertheless, solutions of both systems exhibit partial regularity. In recent work \cite{Hyn}, we have proved analogs of Theorem \ref{mainThm} and \ref{secondThm} for solutions of \eqref{SecondTypePDE}.

\section{Formal computations}
Our first task will  be to briefly derive two simple, yet important integral identities for solutions of \eqref{mainPDE}. They each imply useful local energy estimates. We will establish these identities first for smooth solutions \eqref{mainPDE} and then in the following section we will prove them for weak solutions. In fact, the corresponding estimates dictate the spaces that one may expect weak solutions to belong to; so these estimates actually guide much of the analysis to follow.

\begin{prop}
Let $\psi^*$ be the Legendre transform of $\psi$.   Assume $\bv\in C^\infty(U\times(0,T);\R^m)$ is a solution of \eqref{mainPDE} and $\phi\in C^\infty_c(U\times(0,\infty))$. Then 
\begin{equation}\label{MainIdentity}
\frac{d}{dt}\int_U \psi^*\left(D\psi(\bv)\right)\phi dx+
\int_U \phi DF(D\bv)\cdot D\bv dx=\int_U\left(\psi^*\left(D\psi(\bv)\right)  \phi_t - \bv\cdot DF(D\bv)D\phi \right)dx.
\end{equation}
\end{prop}
\begin{proof} By direct computation, we have 
\begin{align*}
\frac{d}{dt}\int_U \psi^*\left(D\psi(\bv)\right)\phi dx&= \frac{d}{dt}\int_U \phi\left(D\psi(\bv)\cdot \bv -\psi(\bv)\right)dx \\
&=\int_U \phi_t\left(D\psi(\bv)\cdot \bv -\psi(\bv)\right)dx +\int_U \phi\partial_t\left(D\psi(\bv)\right)\cdot \bv dx\\
&= 
\int_U \phi_t\psi^*\left(D\psi(\bv)\right)dx +\int_U \phi\Div[DF(D\bv)]\cdot \bv dx\\
&=  
\int_U \phi_t\psi^*\left(D\psi(\bv)\right)dx - \int_U \bv\cdot DF(D\bv)D\phi dx - \int_U \phi DF(D\bv)\cdot D\bv dx.
\end{align*}
\end{proof}

\par 
Now suppose $O$ is the $m\times n$ matrix of zeros. Notice that we have 
$$
\partial_t\left(D\psi(\bv)-D\psi(0)\right)=\Div\left[DF(D\bv)-DF(O)\right].
$$ 
Consequently, upon subtracting $\psi(0)+D\psi(0)\cdot w$ from $\psi(w)$ and subtracting
$F(O)+DF(O)\cdot M$ from $F(M)$, we may assume without any loss of generality that
\begin{equation}\label{ZeroDeriv}
\psi(0)=|D\psi(0)|=0\quad \text{and}\quad F(O)=|DF(O)|=0.
\end{equation}
With this assumption and \eqref{UnifConv},
\begin{equation}\label{MoreEstPsi}
\begin{cases}
 \frac{1}{2}\theta|w|^2\le \psi(w)\le\frac{1}{2}\Theta |w|^2\\
 \frac{1}{2}\theta|w|^2\le D\psi(w)\cdot w-\psi(w)\le\frac{1}{2}\Theta |w|^2\\
\theta |w|^2\le D\psi(w)\cdot w\le \Theta|w|^2\\
\theta |w|\le |D\psi(w)|\le \Theta |w|.
 \end{cases}
\end{equation}
Analogous inequalities hold for $F$ and $DF$, as well.

\begin{cor}\label{EnergyBound1cor}
Assume $\bv\in C^\infty(U\times(0,T);\R^m)$ is a solution of \eqref{mainPDE}. There is a constant $C=C(\lambda,\Lambda,\theta, \Theta)$ such 
that for each $\eta\in C^\infty_c(U\times(0,\infty))$ with $\eta\ge 0$,
\begin{equation}\label{EnergyBound1}
\max_{0\le t\le T}\int_U\eta^2|\bv|^2 dx + \int^T_0\int_U \eta^2|D\bv|^2dxdt 
\le C\int^T_0\int_U\left(\eta|\eta_t|+|D\eta|^2\right) |\bv|^2dxdt.
\end{equation}
\end{cor}

\begin{proof}  Let $\phi=\eta^2$ in \eqref{MainIdentity}. Employing \eqref{UnifConv} and \eqref{MoreEstPsi} and integrating this identity from $[0,t]$ gives  
\begin{align*}
&\frac{\theta}{2}\int_U\eta(x)^2|\bv(x,t)|^2 dx+\lambda\int^t_0\int_U\eta^2|D\bv|^2 dxds\\
&\hspace{1in}\le \int_U\eta(x)^2\psi^*\left(D\psi(\bv(x,t))\right) dx+\int^t_0\int_U\eta^2 D\bv\cdot DF(D\bv) dxds\\
&\hspace{1in}=\int^t_0\int_U\left(\psi^*\left(D\psi(\bv)\right)  2\eta\eta_t - \bv\cdot DF(D\bv)(2\eta D\eta )\right)dxds\\
&\hspace{1in}\le \int^t_0\int_U\left(\Theta|\bv|^2\eta|\eta_t| +|\bv|\cdot \Lambda|D\bv| 2\eta|D\eta|\right)dxds\\
&\hspace{1in}= \int^t_0\int_U\left(\Theta|\bv|^2\eta|\eta_t| +\left(\frac{2\Lambda}{\sqrt{\lambda}} |D\eta||\bv|\right)\cdot  (\sqrt{\lambda}\eta|D\bv|)\right)dxds\\
&\hspace{1in}\le  \left(\Theta+\frac{2 \Lambda^2}{\lambda}\right)\int^t_0\int_U\left(\eta|\eta_t| +|D\eta|^2\right)|\bv|^2dxds +\frac{\lambda}{2}\int^t_0\int_U\eta^2|D\bv|^2 dxds.
\end{align*}
Consequently, the assertion holds with
$$
C=\frac{\displaystyle\Theta + \frac{2\Lambda^2}{\lambda}}{\frac{1}{2}\min\{\theta,\lambda\}}.
$$
\end{proof}

\par One may interpret identity \eqref{MainIdentity} as the result of multiplying \eqref{mainPDE} by $\phi \bv$ and integrating by parts.  We will now discuss another identity which can be obtained in a similar fashion by multiplying \eqref{mainPDE} by $\phi \bv_t$ and integrating by parts.  

\begin{prop}
Assume $\bv\in C^\infty(U\times(0,T);\R^m)$ is a solution of \eqref{mainPDE} and $\phi\in C^\infty_c(U\times(0,\infty))$. Then 
\begin{equation}\label{MainIdentity2}
\frac{d}{dt}\int_U\phi F(D\bv)dx+\int_U\phi \partial_t(D\psi(\bv))\cdot \bv_t dx
=\int_U\left(\phi_t F(D\bv) - \bv_t \cdot DF(D\bv)D\phi \right)dx.
\end{equation}
\end{prop}
\begin{proof}
We compute
\begin{align*}
\frac{d}{dt}\int_U\phi F(D\bv)dx &=\int_U \phi_t F(D\bv) dx+ \int_U\phi DF(D\bv)\cdot D\bv_tdx\\
&=\int_U \phi_t F(D\bv) dx-\int_U \bv_t \cdot \Div(\phi DF(D\bv))dx\\
&=\int_U \phi_t F(D\bv) dx-\int_U \bv_t \cdot \left(DF(D\bv)D\phi+\phi \Div(DF(D\bv))\right)dx\\
&=\int_U \phi_t F(D\bv) dx-\int_U \bv_t \cdot \left(DF(D\bv)D\phi+\phi\partial_t\left(D\psi(\bv)\right)\right)dx.
\end{align*}
\end{proof}

\begin{cor}\label{EnergyBound2cor}
Assume $\bv\in C^\infty(U\times(0,T);\R^m)$ is a solution of \eqref{mainPDE}. There is a constant $C=C(\lambda,\Lambda,\theta)$ such 
that for each $\eta\in C^\infty_c(U\times(0,\infty))$ with $\eta\ge 0$,
\begin{equation}\label{EnergyBound2}
\max_{0\le t\le T}\int_U\eta^2|D\bv|^2 dx + \int^T_0\int_U \eta^2|\bv_t|^2dxdt 
\le C\int^T_0\int_U\left(\eta|\eta_t|+|D\eta|^2\right) |D\bv|^2dxdt.
\end{equation}
\end{cor}
\begin{proof}
We choose $\phi=\eta^2$ and argue similar to how we did in deriving \eqref{EnergyBound1}.  In particular, 
using \eqref{UnifConv}, \eqref{UnifConv2} and \eqref{ZeroDeriv} and integrating \eqref{MainIdentity2} from $[0,t]$ leads to
\begin{align*}
&\frac{\lambda}{2}\int_U\eta(x)^2|D\bv(x,t)|^2 dx+\theta\int^t_0\int_U\eta^2|\bv_t|^2 dxds\\
&\hspace{1in}\le \int_U\eta(x)^2F(D\bv(x,t)) dx+\int^t_0\int_U\eta^2 D^2\psi(\bv)\bv_t\cdot \bv_t dxds\\
&\hspace{1in}= \int_U\eta(x)^2F(D\bv(x,t)) dx+\int^t_0\int_U\eta^2 \partial_t(D\psi(\bv))\cdot \bv_t dxds\\
&\hspace{1in}=\int^t_0\int_U\left(2\eta\eta_t F(D\bv) - \bv_t \cdot DF(D\bv)2\eta D\eta \right)dxds\\
&\hspace{1in}\le \int^t_0\int_U\left(\Lambda|D\bv|^2\eta|\eta_t| +\eta|\bv_t|\cdot 2\Lambda |D\bv||D\eta|\right)dxds\\
&\hspace{1in}= \Lambda\int^t_0\int_U\left(|D\bv|^2\eta|\eta_t| +\sqrt{\theta}\eta|\bv_t|\cdot \frac{2}{\sqrt{\theta}}|D\bv||D\eta|\right)dxds\\
&\hspace{1in}\le  \Lambda\left(1+\frac{2 }{\theta}\right)\int^t_0\int_U\left(\eta|\eta_t| +|D\eta|^2\right)|D\bv|^2dxds +\frac{\theta}{2}\int^t_0\int_U\eta^2|\bv_t|^2 dxds.
\end{align*}
Therefore, the conclusion holds with
$$
C=\frac{\displaystyle\Lambda\left(1+\frac{2}{\theta}\right)}{\frac{1}{2}\min\{\theta,\lambda\}}.
$$
\end{proof}

\section{Weak Solutions}
In view of estimates \eqref{EnergyBound1} and \eqref{EnergyBound2} and the natural divergence structure of \eqref{mainPDE},  we are lead to a notion of weak solution of \eqref{mainPDE} as specified below. After providing this definition, we will make some crucial observations about the integrability of weak solutions and show that weak solutions satisfy the identities \eqref{MainIdentity} and \eqref{MainIdentity2}.   Then we will establish that weak solutions have a compactness property, which will be vital to our proof of Theorem \ref{mainThm}.

\begin{defn}\label{weakSolnLoc} A measurable mapping $\bv:U\times(0,T)\rightarrow \R^m$ is a {\it weak solution} of \eqref{mainPDE} in $U\times(0,T)$ if $\bv$ satisfies 
\begin{equation}\label{naturalbounds}
\bv\in L^\infty_\text{loc}((0,T); H^1_{\text{loc}}(U;\R^m))\quad\text{and}\quad \bv_t\in L^2_\text{loc}(U\times(0,T);\R^m)
\end{equation}
and
\begin{equation}\label{WeakSolnCond}
\int^T_0\int_U D\psi(\bv)\cdot \bw_t dxdt=\int^T_0\int_U DF(D\bv)\cdot D\bw dxdt,
\end{equation}
for all $\bw\in C^\infty_c(U\times(0,T); \R^m)$.
\end{defn}

\par As mentioned in the introduction, our definition of weak solution definition demands more integrability that what is typically required of solutions to doubly nonlinear parabolic systems (such as those considered in \cite{AltLuck, Vis}).  Instead of Definition \ref{weakSolnLoc}, one could merely ask that a mapping $\bv:U\times(0,T)\rightarrow \R^m$ belong to the space 
$$
\bv\in L^2_\text{loc}((0,T); H^1_{\text{loc}}(U;\R^m))
$$
and satisfy \eqref{WeakSolnCond} for all $\bw\in C^\infty_c(U\times(0,T); \R^m)$. We chose not to do so primarily in view of the identity \eqref{MainIdentity2} and the corresponding energy bound \eqref{EnergyBound2}; these observations are crucial to our proofs of Theorems \ref{mainThm} and \ref{secondThm}.  Moreover, weak solutions as defined above exist.  In Appendix \ref{ExistenceApp}, we show how to construct a weak solution of an initial value problem associated with \eqref{mainPDE} which satisfies the Dirichlet boundary condition via a natural implicit time scheme.

\subsection{Estimates}
We will now proceed to derive some continuity and estimates of weak solutions.  As usual, we will denote $H^{-1}(V;\R^m)$ for the continuous dual space to  $H^1_0(V;\R^m)$ for each open $V\subset U$.

\begin{lem}\label{FirstBasicContLem}
Assume $\bv$ is a weak solution of \eqref{mainPDE} on $U\times(0,T)$, $[t_0,t_1]\subset (0,T)$ and $V\subset\subset U$ is open.  Then 
\begin{equation}\label{DVWeakCont}
\bv: [t_0,t_1]\rightarrow H^{1}(V;\R^m)\text{ is weakly continuous}
\end{equation}
and 
\begin{equation}\label{DpsiVLip}
D\psi\circ \bv: [t_0,t_1]\rightarrow H^{-1}(V;\R^m)\text{ is Lipschitz continuous}.
\end{equation}
\end{lem}
\begin{proof}
In view of \eqref{naturalbounds}, $\bv_t\in L^2([t_0,t_1];L^2(V;\R^m))$. Therefore, $\bv: [t_0,t_1]\rightarrow L^2(V;\R^m)$ is absolutely continuous.  Now let $t_k\in [t_0,t_1]$ with $t_k\rightarrow t$; we clearly have $\bv(\cdot,t_k)\rightarrow \bv(\cdot,t)$ in $L^2(V;\R^m)$. We also have by \eqref{naturalbounds} that $(D\bv(\cdot,t_{k}))_{k\in \N}\subset L^2(V;\Mmn)$ is bounded and hence has a weakly convergent subsequence $(D\bv(\cdot,t_{k_j}))_{j\in \N}$. It is routine to check that the weak limit of this subsequence must be $D\bv(\cdot,t)$. Since this limit is independent of the subsequence, it follows that 
$D\bv(\cdot,t_{k})\rightharpoonup D\bv(\cdot,t)$ in $L^2(V;\Mmn)$. We conclude  \eqref{DVWeakCont}. 

The weak solution condition \eqref{WeakSolnCond} implies 
\begin{equation}\label{WeakSolnCond2}
\frac{d}{dt}\int_U D\psi(\bv(x,t))\cdot \bu(x) dx + \int_U DF(D\bv(x,t))\cdot D\bu(x) dx=0
\end{equation}
in the sense of distributions on $(0,T)$ for each $\bu \in H^1_0(V;\R^m)\subset H^1_0(U;\R^m)$.  In particular, \eqref{WeakSolnCond2} holds for almost every $t\in (0,T)$ (Chapter 3, Lemma 1.1 of \cite{Temam}).   It follows that $D\psi(\bv): [t_0,t_1]\rightarrow H^{-1}(V;\R^m)$ is differentiable almost everywhere on $[t_0,t_1]$ and that 
$$
\|\partial_t\left(D\psi(\bv(\cdot,t))\right)\|_{H^{-1}(V;\R^m)}=\|DF(D\bv(\cdot,t))\|_{L^2(V;\Mmn)}\le  \Lambda\|D\bv(\cdot,t)\|_{L^2(V;\Mmn)}
$$
for almost every $t\in [t_0,t_1]$. In view of \eqref{naturalbounds}, we deduce \eqref{DpsiVLip}. 
\end{proof}
We can also use elliptic regularity results to conclude some integrability of the Hessian of $\bv=(v^1,\dots,v^m)$. Below we use the notation
$$
D^2\bv:=(D^2v^1,\dots,D^2v^m)\in (\M^{n\times n})^m
$$ 
and 
$$
|D^2\bv|^2:=\sum^{m}_{i=1}|D^2v^i|^2=\sum^{m}_{i=1}\sum^n_{j,k=1}(v^i_{x_jx_k})^2.
$$

\begin{lem}\label{H2prop}
Assume $\bv$ is a weak solution of \eqref{mainPDE} on $U\times(0,T)$. Then 
\begin{equation}\label{H2Boundv}
D^2\bv\in L^2_{\text{loc}}(U\times(0,T);(\M^{n\times n})^m).
\end{equation}
In particular, equation \eqref{mainPDE} holds almost everywhere in $U\times(0,T)$.
\end{lem}

\begin{proof}
In view of \eqref{naturalbounds} and the Lipschitz continuity of $D\psi:\R^m\rightarrow\R^m$, 
$$
\partial_t(D\psi\circ\bv)(\cdot,t)=D^2\psi(\bv(\cdot,t))\bv_t(\cdot,t)
$$
in $L^2_{\text{loc}}(U;\R^m)$ for almost every $t\in (0,T)$. In view of \eqref{WeakSolnCond}, 
\begin{equation}\label{MainEqnTimeSlice}
D^2\psi(\bv(\cdot,t))\bv_t(\cdot,t)=\Div DF(D\bv(\cdot,t))
\end{equation}
weakly in $U$ for almost every $t\in (0,T)$. 

Let $W,V\subset U$ be open with $W\subset\subset V\subset\subset U$. 
As $\bv(\cdot,t)$ satisfies \eqref{MainEqnTimeSlice}, the associated $W^{2,2}_{\text{loc}}(U)$ estimates (Proposition 8.6 in \cite{GiaMar} or Theorem 1, Section 8.3 of \cite{Evans}) for uniformly elliptic Euler-Lagrange equations imply $D^2\bv(\cdot, t)\in L^2_{\text{loc}}(U; (\Mnn)^m)$ and
\begin{align*}
\int_{W} |D^2\bv(x,t)|^2dx&\le C\int_V\left(|D^2\psi(\bv(x,t))\bv_t(x,t)|^2+|D\bv(x,t)|^2\right)dx\\
&\le C\int_V\left(\Lambda^2|\bv_t(x,t)|^2+|D\bv(x,t)|^2\right)dx
\end{align*}
for almost every $t\in (0,T)$. Here $C$ is a constant that is independent of $\bv$. The assertion \eqref{H2Boundv} now follows from integrating the above inequality locally in time.

\par Now that we have also established \eqref{H2Boundv}, we can integrate by parts in \eqref{WeakSolnCond} to get 
$$
\int^T_0\int_U \left[\partial_t(D\psi(\bv)) -\Div(DF(D\bv))\right]\cdot \bw dxdt=0,
$$
for all $\bw\in C^\infty_c(U\times(0,T); \R^m)$. Thus $\partial_t(D\psi(\bv)) =\Div(DF(D\bv))$ almost everywhere in $U\times(0,T)$.
\end{proof}
The conclusion of the previous lemma sets up an application of the interpolation of Lebesgue spaces and the Gagliardo-Nirenberg-Sobolev inequality to improve the space time integrability of the gradient of weak solutions.  We refer also to Lemma 5.3 in \cite{DuzMinSte} for a more general result.
\begin{cor}
Assume $\bv$ is a weak solution of \eqref{mainPDE} on $U\times(0,T)$. There exists an exponent $p>2$ such that
\begin{equation}\label{extraIntegrablePee}
D\bv\in L^{p}_{\text{loc}}(U\times(0,T);\Mmn).
\end{equation}
\end{cor}
\begin{proof}
First let $n\ge 3$, $\eta\in C^\infty_c(U)$ and $r\in (2,2^*)$. Here we are using the notation
$$
2^*:=\frac{2n}{n-2}.
$$
Also select $\lambda \in (0,1)$ so that $r=\lambda\cdot 2 + (1-\lambda)\cdot 2^*$; 
that is $\lambda=(2^*-r)/(2^*-2)$. By the interpolation of the Lebesgue spaces and the Gagliardo-Nirenberg-Sobolev inequality, there is a constant $C_0$ depending only on $r$ and $n$ such that
\begin{align*}
\int_{U}|\eta D\bv|^rdx&\le \left(\int_{U}|\eta D\bv|^2dx\right)^\lambda \left(\int_{U}|\eta D\bv|^{2^*}dx\right)^{1-\lambda}\\
&\le C_0\left(\int_{U}|\eta D\bv|^2dx\right)^\lambda \left(\int_{U}|D(\eta D\bv)|^{2}dx\right)^{\frac{2^*}{2}(1-\lambda)}\\
&\le 2^{\frac{2^*}{2}(1-\lambda)}C_0\left(\int_{U}|\eta D\bv|^2dx\right)^\lambda \left(\int_{U}(\eta^2 |D^2\bv|^{2}+|D\bv|^2|D\eta|^2)dx\right)^{\frac{2^*}{2}(1-\lambda)}.
\end{align*}
Here we have suppressed the time dependence of $D\bv$ and $D^2\bv$; however, we note that the above inequality holds for almost every $t\in (0,T)$. 

\par Now we select $r$ such that 
$$
\frac{2^*}{2}(1-\lambda)=\frac{2^*}{2}\frac{r-2}{2^*-2}=1.
$$
Namely, we choose $r=2+\frac{4}{n}\in (2,2^*)$ to get 
$$
\int_{U}|\eta D\bv|^{2+\frac{4}{n}}dx \le 2C\left(\int_{U}|\eta D\bv|^2dx\right)^\lambda \left(\int_{U}(\eta^2 |D^2\bv|^{2}+|D\bv|^2|D\eta|^2)dx\right).
$$
The assertion then follows for $n\ge 3$ by recalling $D\bv\in L^\infty_{\text{loc}}((0,T); L^2_{\text{loc}}(U;\Mmn))$, invoking Lemma \ref{H2prop} and integrating locally in time. For $n=2$, a similar computation can be made to show that $D\bv\in L^{p}_{\text{loc}}(U\times(0,T);\mathbb{M}^{m\times 2})$ for each $p\in [2,4)$.

\par Now let $n=1$, $V\subset\subset U$ be an interval and $[t_0,t_1]\in (0,T)$. By Lemma \ref{H2prop}, 
$$
\bv_x(x,t)=\bv_x(y,t)+\int^x_y\bv_{xx}(z,t)dz
$$
for $x,y\in V$ and almost every $t\in [t_0,t_1]$. It then follows that
$$
|\bv_x(x,t)|^2\le 2\left(|\bv_x(y,t)|^2+|V|\int_V|\bv_{xx}(z,t)|^2dz\right).
$$
Integrating over $y\in V$ gives 
$$
|V||\bv_x(x,t)|^2\le 2\left(\int_V|\bv_x(y,t)|^2dy+|V|^2\int_V|\bv_{xx}(z,t)|^2dz\right).
$$
\par Consequently, $\bv_x\in L^2([t_0,t_1]; L^\infty(V;\R^m))$. Combining with 
$\bv_x\in L^{\infty}([t_0,t_1];L^2(V;\R^m)) $ leads to
\begin{align*}
\int^{t_1}_{t_0}\int_V|\bv_x(x,t)|^4dxdt&=\int^{t_1}_{t_0}\int_V|\bv_x(x,t)|^2|\bv_x(x,t)|^2dxdt\\
&\le \int^{t_1}_{t_0}|\bv_x(\cdot,t)|_{L^\infty(V;\R^m)}^2 \left(\int_V|\bv_x(x,t)|^2dx\right)dt\\
&\le \left(\int^{t_1}_{t_0}|\bv_x(\cdot,t)|_{L^\infty(V;\R^m)}^2dt\right) \left(\displaystyle\text{ess sup}_{t\in [t_0,t_1]}\int_V|\bv_x(x,t)|^2dx\right)\\
&<\infty.
\end{align*}
Thus $\bv_x\in L^4_{\text{loc}}(U\times(0,T);\R^m)$.
\end{proof}

\subsection{Integral identities }
 We now show that identities \eqref{MainIdentity} and \eqref{MainIdentity2} hold for weak solutions of \eqref{mainPDE}.
\begin{prop}\label{IdentityLemma}
Suppose $\phi \in C^\infty_c(U\times(0,T))$ and $\bv$ is a weak solution of \eqref{mainPDE} on $U\times(0,T)$. \\ (i) 
$$
[0,T]\ni t\mapsto \int_U \psi^*\left(D\psi(\bv(x,t))\right)\phi(x,t) dx
$$
is absolutely continuous and \eqref{MainIdentity} holds for almost every $t\in [0,T]$.\\
(ii)
$$
[0,T]\ni t\mapsto \int_U F(D\bv(x,t))\phi(x,t) dx
$$
is absolutely continuous and \eqref{MainIdentity2} holds for almost every $t\in [0,T]$.
\end{prop}

\begin{proof} 1. Suppose that $\phi$ is supported in $V\times(t_0,t_1)\subset \subset U\times (0,T)$ for some open set $V\subset\R^m$.  We have already remarked in our proof of Lemma \ref{FirstBasicContLem} that $\bv: [t_0,t_1]\rightarrow L^2(V;\R^m)$ is absolutely continuous. Since $D\psi:\R^m\rightarrow\R^m$ is Lipschitz and $\psi^*(D\psi(w))=D\psi(w)\cdot w-\psi(w)$, it is routine to check that $[0,T]\ni t\mapsto \int_U \psi^*\left(D\psi(\bv(x,t))\right)\phi(x,t) dx$ is an absolutely continuous function. By Lemma \ref{H2prop},  we also have that equation \eqref{mainPDE} holds almost everywhere in $U\times(0,T)$ for weak solutions. Using virtually the same formal computations we made when deriving \eqref{MainIdentity}, we can show that  \eqref{MainIdentity} holds almost everywhere in $(0,T)$. We conclude assertion $(i)$.

\par 2. So we are left to prove assertion $(ii)$. To this end, we let $u\in C^\infty_c(U)$ be a function that is supported in $V$ and suppose initially that $u\ge 0$. For a given $\bw \in L^2(V;\R^m)$, we also define 
$$
\Phi(\bw):=
\begin{cases}
\displaystyle\int_V u(x)F(D\bw(x))dx, &\quad \bw\in H^1(V;\R^m)\\
+\infty, & \quad \text{otherwise}.
\end{cases}
$$
Note that $\Phi$ is proper, convex and lower-semicontinuous on $L^2(V;\R^m)$. Moreover, a routine computation shows
\begin{align*}
\partial\Phi(\bv(\cdot,t))&:=\left\{\xi\in L^2(V;\R^m): \;\Phi(\bu)\ge \Phi(\bv(\cdot,t))+\int_V\xi\cdot (\bu -\bv(\cdot,t))dx\; \text{ for all}\; \bu\in L^2(V;\R^m)\right\}\\
&=\{-\Div(uDF(D\bv(\cdot,t)))\}\\
&=\{-DF(D\bv(\cdot,t))Du-u\partial_t(D\psi(\bv(\cdot,t)))\}
\end{align*}
for almost every $t\in (0,T)$. Since $-\Div(uDF(D\bv))\in L^2([t_0,t_1]; L^2(V;\R^m)$
and $\bv_t\in  L^2([t_0,t_1]; L^2(V;\R^m)$, it must be that 
$\Phi\circ \bv$ is locally absolutely continuous on $(0,T)$ (Proposition 1.4.4 and Remark 1.4.6 \cite{AGS}). 

\par By the chain rule (Remark 1.4.6 \cite{AGS}), we have by \eqref{WeakSolnCond2}
\begin{align}\label{uTimeIndepIdentity}
\frac{d}{dt} \int_U F(D\bv(x,t))u(x) dx &=\frac{d}{dt} 
\int_V F(D\bv(x,t))u(x) dx \nonumber \\
&= \frac{d}{dt} (\Phi\circ \bv)(\cdot,t)\nonumber \\
&=\int_V\left[-DF(D\bv(\cdot,t))Du-u\partial_t(D\psi(\bv(\cdot,t)))\right]\cdot \bv_tdx\nonumber\\
&=\int_U\left[-DF(D\bv(\cdot,t))Du-u\partial_t(D\psi(\bv(\cdot,t)))\right]\cdot \bv_tdx
\end{align}
for almost every $t\in (0,T)$. In particular, 
\begin{align}\label{uTimeIndepIdentity2}
&\int_U F(D\bv(x,t))u(x) dx = \int_U F(D\bv(x,s))u(x) dx + \\
&\hspace{1.5in}\int^t_s\int_U\left[-DF(D\bv(x,t))Du(x)-u(x)\partial_t(D\psi(\bv(x,t)))\right]\cdot \bv_t(x,t)dxd\tau
\end{align}
for all $t,s\in [0,T]$. 

\par 3. For  $u\in C^\infty_c(U)$ that is not necessarily nonnegative, we may write $u=u^+-u^-.$  Let $(u^\pm)^\epsilon:=\eta^\epsilon*u^\pm$ be the standard mollification of $u^\pm$. Recall that $\eta\in C^\infty_c(B_1(0))$ is a nonnegative, radial function that satisfies $\int_{B_1(0)}\eta(z)dz=1$ and $\eta^\epsilon:=\epsilon^{-n}\eta(\cdot/\epsilon)$. Moreover, for all $\epsilon>0$ sufficiently small, $(u^\pm)^\epsilon\in C^\infty(U)$. Thus \eqref{uTimeIndepIdentity2} holds with $(u^+)^\epsilon$ and $(u^-)^\epsilon$. Subtracting the resulting equalities actually gives 
\begin{align}
&\int_U F(D\bv(x,t))u^\epsilon(x) dx = \int_U F(D\bv(x,s))u^\epsilon(x) dx + \\
 &\hspace{1.7in}\int^t_s\int_U\left[-DF(D\bv(x,t))Du^\epsilon(x)-u^\epsilon(x)\partial_t(D\psi(\bv(x,t)))\right]\cdot \bv_t(x,t)dxd\tau
\end{align}
for $t,s\in [0,T]$.  Therefore, sending $\epsilon\rightarrow 0^+$ allows us to recover \eqref{uTimeIndepIdentity2} for $u$ without making any restrictions on the sign of $u$. In particular, \eqref{uTimeIndepIdentity} holds for any $u\in C_c^\infty(U)$.

\par 4. Now let $\eta\in C^\infty_c(U)$ be nonnegative and $\eta|_V\equiv 1$.   We will use this function to show that $\bv: (0,T)\rightarrow H^1(V;\R^m)$ is continuous. Recall that we have already shown that this mapping is weakly continuous in \eqref{DVWeakCont}. Moreover, \eqref{uTimeIndepIdentity} implies that $(0,T)\ni t\mapsto \int_U\eta(x) F(D\bv(x,t))dx$ is continuous.  Now let $t_k\in (0,T)$ with $t_k\rightarrow t\in (0,T)$. The uniform convexity of $F$ gives
\begin{align*}
 \int_U F(D\bv(x,t_k))\eta(x)dx 
 & \ge  \int_U F(D\bv(x,t))\eta(x)dx  \\
 & +  \int_U DF(D\bv(x,t))\cdot(D\bv(x,t_k)-D\bv(x,t)) \eta(x)dx \\
 &+\frac{\lambda}{2}\int_U |D\bv(x,t_k)-D\bv(x,t)|^2\eta(x)dx\\
  & \ge  \int_U F(D\bv(x,t))\eta(x)dx  \\
 & +  \int_U DF(D\bv(x,t))\cdot(D\bv(x,t_k)-D\bv(x,t)) \eta(x)dx \\
 &+\frac{\lambda}{2}\int_V |D\bv(x,t_k)-D\bv(x,t)|^2dx.
\end{align*}
Sending $k\rightarrow\infty$, 
$$
\limsup_{k\rightarrow \infty}\int_V |D\bv(x,t_k)-D\bv(x,t)|^2dx=0.
$$
This shows $\bv: (0,T)\rightarrow H^1(V;\R^m)$ is continuous.

\par 5.  Set 
$f(t):=  \int_U F(D\bv(x,t))\phi(x,t) dx$ and let $h\neq 0$. Observe 
\begin{align*}
\frac{f(t+h)-f(t)}{h}&=
\int_U\frac{F(D\bv(x,t+h))\phi(x,t+h)-F(D\bv(x,t))\phi(x,t)}{h}dx\\
&=\int_U\phi(x,t)\left[\frac{F(D\bv(x,t+h))-F(D\bv(x,t))}{h}\right]dx \\
&\hspace{.5in}+ \int_U F(D\bv(x,t+h))\left[\frac{\phi(x,t+h)-\phi(x,t)}{h} \right]dx.
\end{align*}
By parts 2 and 3 of this proof, 
\begin{align*}
\lim_{h\rightarrow 0}\int_U\phi(x,t)\left[\frac{F(D\bv(x,t+h))-F(D\bv(x,t))}{h}\right]dx  = 
\hspace{1in} \\
-\int_U\bv_t(x,t)\cdot DF(D\bv(x,t))D\phi(x,t) dx-\int_U\phi(x,t) \partial_t(D\psi(\bv(x,t)))\cdot \bv_t(x,t) dx
\end{align*}
for almost every $t\in (0,T)$.  From part 4 of this proof, 
\begin{align*}
\lim_{h\rightarrow 0}\int_U F(D\bv(x,t+h))\left[\frac{\phi(x,t+h)-\phi(x,t)}{h} \right]dx=\int_U F(D\bv(x,t))\phi_t(x,t)dx
\end{align*}
for every $t\in (0,T)$.  Combining these limits completes a proof of \eqref{MainIdentity2}. Finally, we note that if \eqref{MainIdentity2} holds then $f$ is absolutely continuous as each term in \eqref{MainIdentity} aside from the time derivative belongs to $L^1[0,T]$.
\end{proof}
\begin{cor}\label{EnergyBound1Cor}
Every weak solution of \eqref{mainPDE} in $U\times(0,T)$ satisfies inequalities \eqref{EnergyBound1} and \eqref{EnergyBound2}. 
\end{cor}

\subsection{Fractional time differentiability}
In part 4 of the proof of Proposition \ref{IdentityLemma}, we showed that for each weak solution $\bv$ of \eqref{mainPDE} in $U\times (0,T)$, $\bv:(0,T)\rightarrow H^1(V;\R^m)$ is continuous for every $V\subset\subset U$.  This strengthened our previous assertion \eqref{DVWeakCont}. Now we will build on these observations and establish a type of fractional time differentiability of $D\bv$. The following estimate will be crucial to our proof of Theorem \ref{secondThm}.
\begin{prop}
Assume $\bv$ is a weak solution of \eqref{mainPDE} in $U\times (0,T)$, and let $p>2$ be the exponent in \eqref{extraIntegrablePee}.  For each open $V\subset\subset U$ and $[t_0,t_1]\in (0,T)$, there is a constant $C$ such that 
\begin{equation}\label{FractionDVDeriv}
\int^{t_1}_{t_0}\int_V|D\bv(x,t+h)-D\bv(x,t)|^2dxdt\le C |h|^{\frac{1}{2}-\frac{1}{p}}
\end{equation} 
for $0<|h|<\min\{1,t_0,T-t_1\}$. 
\end{prop}

\begin{proof}
Let $u\in C^\infty_c(U)$ be a nonnegative function satisfying $u\equiv 1$ on $V$.  We have from \eqref{uTimeIndepIdentity} that 
$$
\int_Uu(x) F(D\bv(x,t+h))dx
\le \int_Uu(x) F(D\bv(x,t))dx- \int^{t+h}_t\int_U\bv_t \cdot DF(D\bv)Du dxds
$$
for $t\in [t_0,t_1]$ and $h\in (0,\min\{1,t_0,T-t_1\})$. By the uniform convexity of $F$
\begin{align}\label{Step1Fractional}
\int^{t+h}_t\int_U\bv_t \cdot DF(D\bv)Du dxds&\ge \int_Uu(x)(F(D\bv(x,t+h))-F(D\bv(x,t)) )dx\nonumber\\
&\ge \int_Uu(x)DF(D\bv(x,t))\cdot(D\bv(x,t+h)-D\bv(x,t))dx\nonumber \\
&\quad + \frac{\lambda}{2}\int_Uu(x)|D\bv(x,t+h)-D\bv(x,t)|^2dx\nonumber\\
&\ge -\int_V\Div(u(x)DF(D\bv(x,t)))\cdot(\bv(x,t+h)-\bv(x,t))dx \nonumber\\
&\quad + \frac{\lambda}{2}\int_V|D\bv(x,t+h)-D\bv(x,t)|^2dx.
\end{align}

\par By \eqref{naturalbounds} and \eqref{extraIntegrablePee}, $|\bv_t||D\bv|\in L^{\frac{2p}{p+2}}_{\text{loc}}(U\times (0,T);\R^m)$.  Therefore,
\begin{align}\label{Step2Fractional}
\int^{t+h}_t\int_U\bv_t \cdot DF(D\bv)Du dxds&\le C_0\int^{t+h}_t\int_V|\bv_t||D\bv|dxds \nonumber\\
&\le C_0\left(\int^{t_1}_{t_0}\int_V\left(|\bv_t||D\bv|\right)^{\frac{2p}{p+2}}dxds\right)^{\frac{1}{2}+\frac{1}{p}}\left(|V|h\right)^{\frac{1}{2}-\frac{1}{p}} \nonumber\\
&\le C_0|V|^{\frac{1}{2}-\frac{1}{p}}\left(\int^{t_1}_{t_0}\int_V\left(|\bv_t||D\bv|\right)^{\frac{2p}{p+2}}dxds\right)^{\frac{1}{2}+\frac{1}{p}}h^{\frac{1}{2}-\frac{1}{p}}.\quad
\end{align}
Here $C_0$ depends on $\Lambda$ and $\|Du\|_{L^\infty(U;\R^n)}$.  We can also use \eqref{naturalbounds} to conclude
\begin{align}\label{Step3Fractional}
&\int_V\Div(u(x)DF(D\bv(x,t)))\cdot(\bv(x,t+h)-\bv(x,t))dx  \hspace{2in}\nonumber\\
&\hspace{1in}\le\left(\int_V|\Div(uDF(D\bv))|^2dx\right)^{1/2}\left(\int_V|\bv(x,t+h)-\bv(x,t)|^2dx\right)^{1/2}\nonumber\\
&\hspace{1in}\le\left(\int_V\left|Du\cdot DF(D\bv)+ u\partial_t(D\psi(\bv))\right|^2dx\right)^{1/2}\left(h\int^{t+h}_{t}\int_V|\bv_t|^2dxds\right)^{1/2}\nonumber\\
&\hspace{1in}\le C_1\left(\int_V(|D\bv|^2+ |\bv_t|^2)dx\right)^{1/2}\left(\int^{t_1}_{t_0}\int_V|\bv_t|^2dxds\right)^{1/2}h^{1/2},
\end{align}
for a constant $C_1$ depending on $\Lambda,\Theta$ and $\|u\|_{L^\infty(U)}+\|Du\|_{L^\infty(U;\R^n)}$.

\par Combining \eqref{Step1Fractional}, \eqref{Step2Fractional} and \eqref{Step3Fractional} gives 
\begin{align*}
&\frac{\lambda}{2}\int^{t_1}_{t_0}\int_V|D\bv(x,t+h)-D\bv(x,t)|^2dxdt \\
& \hspace{1in}\le C_0(t_1-t_0)|V|^{\frac{1}{2}-\frac{1}{p}}\left(\int^{t_1}_{t_0}\int_V\left(|\bv_t||D\bv|\right)^{\frac{2p}{p+2}}dxds\right)^{\frac{1}{2}+\frac{1}{p}}h^{\frac{1}{2}-\frac{1}{p}}\\
& \hspace{1in}+ C_1\left\{\int^{t_1}_{t_0}\left(\int_V(|D\bv|^2+ |\bv_t|^2)dx\right)^{1/2}dt\right\}\left(\int^{t_1}_{t_0}\int_V|\bv_t|^2dxds\right)^{1/2}h^{1/2} \\
& \hspace{1in}\le C_0(t_1-t_0)|V|^{\frac{1}{2}-\frac{1}{p}}\left(\int^{t_1}_{t_0}\int_V\left(|\bv_t||D\bv|\right)^{\frac{2p}{p+2}}dxds\right)^{\frac{1}{2}+\frac{1}{p}}h^{\frac{1}{2}-\frac{1}{p}}\\
& \hspace{1in}+ C_1|t_1-t_0|^{1/2}\left(\int^{t_1}_{t_0}\int_V(|D\bv|^2+ |\bv_t|^2)dxdt\right)^{1/2}\left(\int^{t_1}_{t_0}\int_V|\bv_t|^2dxds\right)^{1/2}h^{1/2}\\
& \hspace{1in} \le \left\{C_0(t_1-t_0)|V|^{\frac{1}{2}-\frac{1}{p}}\left(\int^{t_1}_{t_0}\int_V\left(|\bv_t||D\bv|\right)^{\frac{2p}{p+2}}dxds\right)^{\frac{1}{2}+\frac{1}{p}}  \right. \\
&\hspace{1in}\left.+ C_1|t_1-t_0|^{1/2}\left(\int^{t_1}_{t_0}\int_V(|D\bv|^2+ |\bv_t|^2)dxdt\right)^{1/2}\left(\int^{t_1}_{t_0}\int_V|\bv_t|^2dxds\right)^{1/2}\right\}h^{\frac{1}{2}-\frac{1}{p}}.
\end{align*}
This computation establishes \eqref{FractionDVDeriv} for positive $h$.   A similar argument establishes  \eqref{FractionDVDeriv} for negative $h$. We leave the details to the reader. 
\end{proof}
It now follows fairly routinely that $D\bv$ is fractionally differentiable with respect to time as exhibited in \eqref{FractionDVDeriv2} below. The following assertion can be found in Proposition 3.4  \cite{DuzMin} or Proposition 2.19 of \cite{DuzMinSte}.
\begin{cor}\label{TimeDiffDVCor}
Assume $\bv$ is a weak solution of \eqref{mainPDE} in $U\times (0,T)$ and $p>2$ is the exponent in \eqref{extraIntegrablePee}.  For each open $V\subset\subset U$, $[t_0,t_1]\in (0,T)$, and
\begin{equation}\label{BetaInterval}
\beta\in \left(0,\frac{1}{2}-\frac{1}{p}\right),
\end{equation}
there is constant $\kappa=\kappa(p,\beta,n,t_0,t_1, V)>0$ such that 
\begin{equation}\label{FractionDVDeriv2}
\int^{t_1}_{t_0}\int^{t_1}_{t_0}\int_V\frac{|D\bv(x,t)-D\bv(x,s)|^2}{|t-s|^{1+\beta}}dxdtds\le \kappa\left(C+\|D\bv\|^2_{L^2(V\times[t_0,t_1])}\right).
\end{equation} 
Here $C$ is the constant in \eqref{FractionDVDeriv}.
\end{cor}

\subsection{Compactness}
Now we will discuss the compactness properties of weak solutions.  Roughly speaking, we will show that any ``bounded" sequence of weak solutions to systems of the form \eqref{mainPDE} has a subsequence that converges ``strongly" to another weak solution of this type of system.  Our mail tools will be the identities \eqref{MainIdentity} and \eqref{MainIdentity2},  the energy estimates \eqref{EnergyBound1} and \eqref{EnergyBound2} and a compactness result due to J. P. Aubin \cite{Aubin}.  

\begin{prop}\label{compactProp}
Let $\psi^k\in C^1(\R^m)$ satisfy \eqref{UnifConv} and $F^k\in C^1(\Mmn)$ 
satisfy \eqref{UnifConv2} for each $k\in \N$.  Suppose $(\bv^k)_{k\in \N}$ is a sequence of weak solutions of
$$
\partial_t\left(D\psi^k(\bv^k)\right)=\normalfont{\Div} DF^k(D\bv^k)
$$
in $U\times (0,T)$ and assume 
\begin{equation}\label{UnifL2Bound}
\sup_{k\in\N}\int^T_0\int_U(|\bv^k|^2+|D\bv^k|^2)dxdt <\infty.
\end{equation}
Then there is a subsequence $(\bv^{k_j})_{j\in \N}$ and $\bv\in H^1(U\times (0,T);\R^m)$ such that
for each open $V\subset\subset U$ and interval $[t_0,t_1]\subset (0,T)$, 
\begin{equation}\label{StrongConv}
\bv^{k_j}\rightarrow \bv \; \text{in}\;C([t_0,t_1]; H^1(V;\R^m))
\end{equation}
and 
\begin{equation}\label{StrongConvVtee}
\bv^{k_j}_t\rightarrow \bv_t \; \text{in}\;L^2(U\times (0,T);\R^m).
\end{equation}
Moreover, there is $\psi\in C^1(\R^m)$ that satisfies \eqref{UnifConv} and $F\in C^1(\Mmn)$ that satisfies \eqref{UnifConv2} for which $\bv$ is a weak solution of 
\eqref{mainPDE} in $U\times (0,T)$.
\end{prop}

\begin{proof}
1. By hypothesis, we have that
$$
|D\psi^k(z_1)-D\psi^k(z_2)|\le \Theta|z_1-z_2|\quad (z_1,z_2\in \R^m)
$$
for each $k\in \N$. As noted above, we may assume without loss of generality that $\psi^k$ satisfies \eqref{ZeroDeriv}. Upon making this assumption, we have that the sequence $(D\psi^k)_{k\in \N}$ is both equicontinuous and locally uniformly bounded on $\R^m$.  By the Arzel\`a-Ascoli Theorem, there is a subsequence  $(\psi^{k_j})_{j\in \N}$ and $\psi \in C^{1}(\R^m)$ such that $\psi^{k_j} \rightarrow \psi$ and $D\psi^{k_j} \rightarrow D\psi$ locally uniformly on $\R^m$.  Moreover, $\psi$ satisfies \eqref{UnifConv}.  Similarly, there is a subsequence $(F^{k_j})_{j\in \N}$ and $F\in C^1(\Mmn)$ such that $F^{k_j} \rightarrow F$ and $DF^{k_j} \rightarrow DF$ locally uniformly on $\Mmn$ and $F$ satisfies  \eqref{UnifConv2}.

\par  Also note by \eqref{UnifL2Bound} that there is $\bv\in L^2((0,T); H^1(U;\R^m))$ and a subsequence of $(\bv^{k_j})_{j\in \N}$ (not relabeled) such that 
$$
\bv^{k_j}\rightharpoonup \bv
$$
in $L^2((0,T); H^1(U;\R^m))$.  In view of Corollary \ref{EnergyBound1Cor}, $\bv$ also satisfies \eqref{naturalbounds}. If we can establish \eqref{StrongConv} for each open $V\subset\subset U$ and an interval $[t_0,t_1]\subset (0,T)$, we can then pass to the limit in
\begin{equation}\label{WeakEqnKay}
\int^T_0\int_U D\psi^k(\bv^k)\cdot \bw_t dxdt=\int^T_0\int_U DF^k(D\bv^k)\cdot D\bw dxdt,
\end{equation}
as $k=k_j\rightarrow \infty$ for each $\bw\in C^\infty_c(U\times(0,T); \R^m)$. It would then follow that $\bv$ is necessarily a weak solution of \eqref{mainPDE} in $U\times (0,T)$. Thus, we focus on proving \eqref{StrongConv}; along the way we will also verify \eqref{StrongConvVtee}. We finally note that since $\partial U$ is smooth, it suffices to verify 
\eqref{StrongConv} for $V$ with smooth boundary $\partial V$.

\par 2.  By Corollary \ref{EnergyBound1Cor} and our assumptions on $\psi^k$ and $F^k$,
\begin{equation}\label{BasicBoundCompacnt}
\sup_{j\in \N}\left\{\max_{t\in [t_0,t_1]}\int_V(|\bv^{k_j}(x,t)|^2 +|D\bv^{k_j}(x,t)|^2)dx + \int^{t_1}_{t_0}\int_V |\bv^{k_j}_t|^2dxdt 
\right\}<\infty.
\end{equation}
This bound implies that $\bv^{k_j} : [t_0,t_1]\rightarrow L^2(V;\R^m)$ is uniformly equicontinuous and \newline $(\bv^{k_j}(\cdot,t))_{j\in \N}\subset H^1(V;\R^m)$ is uniformly bounded independently of $t\in [t_0,t_1]$. 
Since $H^{1}(V;\R^m)\newline\subset L^2(V;\R^m)$ with compact embedding, there is a further subsequence (not relabeled) such that 
\begin{equation}\label{halfStrongConv}
\bv^{k_j}\rightarrow \bv \; \text{in}\;C([t_0,t_1]; L^2(V;\R^m)).
\end{equation}
This compactness is due to a well known result of J. P. Aubin \cite{Aubin}; see also \cite{Simon} for an extended discussion.

\par Since $(\bv^{k_j}(\cdot,t))_{j\in \N}$ is bounded in $H^1(V;\R^m)$ uniformly in $t\in[t_0,t_1]$, it follows from \eqref{halfStrongConv} that 
\begin{equation}\label{UnifWeakLim}
D\bv^{k_j}(\cdot,t)\rightharpoonup D\bv(\cdot,t)\; \text{in $L^2(V ;\Mmn)$ uniformly for}\;t\in[t_0,t_1].
\end{equation}
We also have (up to a subsequence) that 
$$
DF(D\bv^{k_j})\rightharpoonup \xi\;\; \text{in}\;\;L^2(V\times [t_0,t_1];\Mmn).
$$
Combined with \eqref{halfStrongConv}, we can send $k=k_j\rightarrow\infty$ in \eqref{WeakEqnKay} to find 
\begin{equation}\label{DistrEqn}
\partial_t(D\psi(\bv))=\Div(\xi).
\end{equation}
\par 3. We will now use the identity \eqref{MainIdentity}. Suppose $\phi\in C^\infty_c(U\times(0,T))$ is supported in $V\times (t_0,t_1)$ and is nonnegative. Then for $t,s \in [t_0,t_1]$ 
\begin{align}\label{veekayJayIdentity}
 \int_U (\psi^{k_j})^*\left(D\psi^{k_j}(\bv^{k_j}(x,t))\right)\phi(x,t) dx+\int^t_s\int_U \phi DF^{k_j}(D\bv^{k_j})\cdot D\bv^{k_j} dxd\tau = \hspace{1in} \nonumber \\
\int_U (\psi^{k_j})^*\left(D\psi^{k_j}(\bv^{k_j}(x,s))\right)\phi(x,s) dx+\int^t_s\int_U\left((\psi^{k_j})^*\left(D\psi^{k_j}(\bv^{k_j})\right)  \phi_t - \bv^{k_j}\cdot DF^{k_j}(D\bv^{k_j})D\phi \right)dxd\tau.
\end{align}
Sending $j\rightarrow \infty$ gives 
\begin{align*}
 \int_U \psi^*\left(D\psi(\bv(x,t))\right)\phi(x,t) dx+\lim_{j\rightarrow\infty}\int^t_s\int_U \phi DF^{k_j}(D\bv^{k_j})\cdot D\bv^{k_j} dxd\tau = \hspace{.5in}  \\
\int_U \psi^*\left(D\psi(\bv(x,s))\right)\phi(x,s) dx+\int^t_s\int_U\left(\psi^*\left(D\psi(\bv)\right)  \phi_t - \bv\cdot \xi D\phi \right)dxd\tau.
\end{align*}
\par On the other hand, we can use equation \eqref{DistrEqn} to derive the identity 
\begin{align*}
 \int_U \psi^*\left(D\psi(\bv(x,t))\right)\phi(x,t) dx+\int^t_s\int_U \phi \xi \cdot D\bv dxd\tau = \hspace{1in} \\
\int_U \psi^*\left(D\psi(\bv(x,s))\right)\phi(x,s) dx+\int^t_s\int_U\left(\psi^*\left(D\psi(\bv)\right)  \phi_t - \bv\cdot \xi D\phi \right)dxd\tau.
\end{align*}
A proof of this identity can be made similar to the one given above for Proposition \ref{IdentityLemma}. Therefore, 
$$
\lim_{j\rightarrow\infty}\int^t_s\int_U \phi DF^{k_j}(D\bv^{k_j})\cdot D\bv^{k_j} dxd\tau=\int^t_s\int_U \phi \xi \cdot D\bv dxd\tau.
$$
Consequently,
\begin{align*}
&\lim_{j\rightarrow \infty}\left\{\lambda\int^t_s\int_V\phi|D\bv^{k_j}-D\bv|^2dxd\tau\right\} \hspace{3in}\\
&\hspace{1in}\le \lim_{j\rightarrow \infty}\int^t_s\int_V\phi(DF^{k_j}(D\bv^{k_j})-DF^{k_j}(D\bv))\cdot (D\bv^{k_j}-D\bv)dxd\tau\\
&\hspace{1in}=0.
\end{align*}
It follows that $D\bv^{k_j} \rightarrow D\bv$ in $L^2_{\text{loc}}(V\times(t_0,t_1); \Mmn)$ and that $\xi=DF(\bv)$.  In view of  \eqref{DistrEqn} and the arbitrariness of $V$ and $[t_0,t_1]$, we conclude that $\bv$ is a weak solution of \eqref{mainPDE} in $U\times(0,T)$. 

\par 4. By \eqref{BasicBoundCompacnt} and \eqref{halfStrongConv}, $\bv^{k_j}_t\rightharpoonup \bv_t$ in $L^2_{\text{loc}}(V\times(t_0,t_1); \R^m)$. We have from part $(ii)$ of Proposition \ref{IdentityLemma} that
\begin{align}\label{LastThingCompacntess}
\int_U\phi(x,t)F^{k_j}(D\bv^{k_j}(x,t))dx+\int^t_0\int_U\phi D^2\psi^{k_j}(\bv^{k_j})\bv^{k_j}_t\cdot \bv_t^{k_j} dxds\nonumber\\
=\int^t_0\int_U\left(\phi_t F(D\bv^{k_j}) - \bv^{k_j}_t \cdot DF^{k_j}(D\bv^{k_j})D\phi \right)dxds.
\end{align}
for each nonnegative $\phi\in C^\infty_c(U\times(0,T))$. Letting $j\rightarrow \infty$ gives 
\begin{align*}
\liminf_{j\rightarrow\infty}\int_U\phi(x,t)F^{k_j}(D\bv^{k_j}(x,t))dx+\liminf_{j\rightarrow\infty}\int^t_0\int_U\phi D^2\psi^{k_j}(\bv^{k_j})\bv^{k_j}_t\cdot \bv_t^{k_j} dxds\\
\le \int^t_0\int_U\left(\phi_t F(D\bv) - \bv_t \cdot DF(D\bv)D\phi \right)dxds.
\end{align*}
On the other hand, since $\bv$ is a weak solution 
\begin{align*}
\int_U\phi(x,t)F(D\bv(x,t))dx+\int^t_0\int_U\phi D^2\psi(\bv)\bv_t\cdot \bv_t dxds\\
=\int^t_0\int_U\left(\phi_t F(D\bv) - \bv_t \cdot DF(D\bv)D\phi \right)dxds.
\end{align*}
\par Therefore, 
$$
\liminf_{j\rightarrow\infty}\int_U\phi(x,t)F^{k_j}(D\bv^{k_j}(x,t))dx= \int_U\phi(x,t)F(D\bv(x,t))dx
$$
and
$$
\liminf_{j\rightarrow\infty}\int^t_0\int_U\phi D^2\psi^{k_j}(\bv^{k_j})\bv^{k_j}_t\cdot \bv_t^{k_j} dxds = \int^t_0\int_U\phi D^2\psi(\bv)\bv_t\cdot \bv_t dxds.
$$
for each $t\in (0,T)$. It is now routine to check using the uniform convexity of $\psi^{k_j}$ and $F^{k_j}$, \eqref{halfStrongConv} and \eqref{UnifWeakLim} that $\bv^{k_j}(\cdot,t)\rightarrow \bv(\cdot,t)$ in $H^1_{\text{loc}}(U;\R^m)$ for each $t\in (0,T)$ and that $\bv^{k_j}_t\rightarrow \bv_t$ in $L^2_{\text{loc}}(U\times(0,T); \R^m)$. In particular, we conclude \eqref{StrongConvVtee}.  

\par 5. With these strong convergence assertions and \eqref{LastThingCompacntess}, we have that the sequence of functions
$$
[0,T]\ni t\mapsto \int_U\phi(x,t)F^{k_j}(D\bv^{k_j}(x,t))dx
$$
is uniformly equicontinuous. Indeed, the arguments we made above imply that
\begin{align*}
&\frac{d}{dt}\int_U\phi F^{k_j}(D\bv^{k_j})dx=\int_U\left(\phi_t F^{k_j}(D\bv^{k_j}) - \bv^{k_j}_t \cdot DF^{k_j}(D\bv^{k_j})D\phi \right)dxds
 \\
&\hspace{2in}-\int_U\phi D^2\psi^{k_j}(\bv^{k_j})\bv^{k_j}_t\cdot \bv_t^{k_j} dxds
\end{align*}
converges in $L^1([0,T])$ and so is uniformly integrable. Combined with the pointwise convergence of $\bv^{k_j}(\cdot,t)\rightarrow \bv(\cdot,t)$ in $H^1_{\text{loc}}(U;\R^m)$, we conclude
$$
\lim_{j\rightarrow\infty}\int_U\phi(x,t)F^{k_j}(D\bv^{k_j}(x,t))dx=\int_U\phi(x,t)F(D\bv(x,t))dx
$$
uniformly in $[0,T]$. Recalling \eqref{UnifWeakLim} and that $F$ is uniformly convex, we finally deduce \eqref{StrongConv}.
\end{proof}

\section{Partial regularity}
In this section, we will complete the main goal of this work which is to verify Theorems \ref{mainThm} and \ref{secondThm}.  In order to establish Theorem \ref{mainThm}, we will use the ideas that go into proving Proposition \ref{compactProp} to deduce a local H\"older regularity criterion for weak solutions. 
Then we will use a standard Poincar\'e inequality and Lebesgue differentiation to show almost everywhere H\"older regularity of the spatial gradient of weak solutions.  As for Theorem \ref{secondThm}, we will employ a more refined Poincar\'e inequality that is based on the fractional time differentiability of weak solutions that was asserted in Corollary \ref{TimeDiffDVCor}.

\subsection{A local regularity criterion}
Let us denote a parabolic cylinder of radius $r>0$ centered at $(x,t)$ as
$$
Q_r(x,t):=B_r(x)\times(t-r^2/2, t+r^2/2)
$$
and the average of $\bw$ over $Q_r=Q_r(x,t)$ as
$$
\bw_{Q_r}:=\ffint\;\bw=\frac{1}{|Q_r|}\iint_{Q_r}\bw(y,s)dyds.
$$
A quantity that will be of great utility to us is 
$$
E(x,t,r):=\ffint\;\; \left|\frac{\bv - (\bv)_{Q_r}- (D\bv)_{Q_r}(y-x) }{r}\right|^2dyds + \ffint|D\bv - (D\bv)_{Q_r}|^2 dyds.
$$
Here $\bv$ is a weak solution of  \eqref{mainPDE} in $U\times(0,T)$ and $Q_r=Q_r(x,t)\subset U\times(0,T)$ with $r>0$.  

\par We will now derive an important decay property of $E$.

\begin{lem}\label{BlowUpLemma}
Assume $\bv$ is a weak solution of \eqref{mainPDE} in $U\times(0,T)$. For each $L>0$, there are $\epsilon, \vartheta, \rho\in (0,1/2)$ for which 
\begin{equation}\label{BlowLemHypoth}
\begin{cases}
Q_r(x,t)\subset U\times(0,T), \quad r<\rho\\
|(\bv)_{Q_r}|, |(D\bv)_{Q_r}|\le L\\
E(x,t,r)\le \epsilon^2
\end{cases} \nonumber
\end{equation}
implies 
$$
E(x,t,\vartheta r)\le \frac{1}{2}E(x,t,r).
$$
\end{lem}

\begin{proof}
1. We will argue by contradiction. If the assertion is false, there is $L_0>0$ and sequences $(x_k,t_k)\in U\times(0,T)$, $\epsilon_k\rightarrow 0$, $\vartheta_k\equiv\vartheta\in(0,1/2)$ (chosen below), and
$r_k\rightarrow 0$ as $k\rightarrow +\infty$  such that
\begin{equation}\label{ContraUno}
\begin{cases}
Q_{r_k}(x_k,t_k)\subset U\times(0,T)\\
|(\bv)_{Q_{r_k}}|, |(D\bv)_{Q_{r_k}}|\le L_0\\
E(x_k,t_k,r_k)=\epsilon_k^2
\end{cases}
\end{equation}
while
\begin{equation}\label{ContradictionMaybe}
E(x_k,t_k,\vartheta r_k)> \frac{1}{2}\epsilon_k^2.
\end{equation}

\par Define the sequence of mappings
\begin{equation}\label{BlowupSeq}
\bv^k(y,s):=\frac{\bv(x_k+r_ky,t_k + r_k^2s) - (\bv)_{Q_{r_k}}  -(D\bv)_{Q_{r_k}}r_ky}
{\epsilon_k r_k},\nonumber
\end{equation}
for $(y,s)\in Q_1:=Q_1(0,0)$ and $k\in \N$.  As $E(x_k,t_k,r_k)=\epsilon_k^2$, $\bv^k$ satisfies 
\begin{equation}\label{vkkbounds}
\fffint|\bv^k|^2dyds + \fffint|D\bv^k|^2dyds  =1
\end{equation}
for each $k\in \N$.  Moreover, $\bv^k$ is a weak solution of the PDE
\begin{equation}\label{BlowUpEqn}
\partial_s\left[D_w\psi^k(y,\bv^k)\right] = \Div_y\left[DF^k(D\bv^k)\right]
\end{equation}
in $Q_1$.  Here 
$$
\psi^k(y,w):=\frac{\psi(a_k+r_k M_k y +\epsilon_k r_k w)-\psi(a_k+r_k M_k y)-D\psi(a_k+r_k M_k y)\cdot \epsilon_kr_k w}{\epsilon_k^2r_k^2},
$$ 
and
$$
F^k(\xi):=\frac{F( M_k +\epsilon_k \xi)-F( M_k)-DF(M_k)\cdot \epsilon_k \xi}{\epsilon_k^2}.
$$
We are now using the notation $a_k:= (\bv)_{Q_{r_k}}\in \R^m$, $M_k:= (D\bv)_{Q_{r_k}}\in \Mmn$; and in view of \eqref{ContraUno}, these sequences are bounded.  Without any loss of generality, we assume $a_k\rightarrow a\in \R^m$ and $M_k\rightarrow M\in\Mmn$. 

\par 2. Observe that for each $y\in \R^n$, $w\mapsto\psi^k(y,w)$ is uniformly convex and satisfies \eqref{UnifConv}. Moreover, 
$$
\begin{cases}
\psi^k(y,w)\rightarrow \frac{1}{2}D^2\psi(a)w\cdot w\\
D_w\psi^k(y,w)\rightarrow D^2\psi(a)w
\end{cases}
$$
for each $(y,w)\in \R^n\times\R^m$ as $k\rightarrow \infty$.  Likewise, $F^k$ satisfies \eqref{UnifConv} and
$$
\begin{cases}
F^k(\xi)\rightarrow \frac{1}{2}D^2F(M)\xi\cdot \xi\\
DF^k(\xi)\rightarrow D^2F(M)\xi
\end{cases}
$$
for every $\xi\in \Mmn$ as $k\rightarrow \infty$. Furthermore, $\psi^k(y,\cdot)$ and $F^k$ satisfies \eqref{ZeroDeriv} for each $y\in \R^n$. 

\par By \eqref{vkkbounds} and a minor variant of Proposition \ref{compactProp}, there is $\bw\in L^2\left([-1/2,1/2]; H^1(B_1;\R^m)\right)$ and a subsequence $\{\bv^{k_j}\}_{j\in \N}$ satisfying
\begin{equation}\label{CompactnessBlow}
\bv^{k_j}\rightarrow \bw\quad \text{in}\quad C\left(\left[t_0,t_1\right]; H^1(B_R;\R^m)\right)
\end{equation}
for each $[t_0,t_1]\subset (-1/2,1/2)$ and $R\in (0,1)$.
Moreover,
\begin{equation}\label{WeakL22bounds}
\fffint|\bw|^2dyds + \fffint|D\bw|^2dyds = 1.
\end{equation}
Recall that the weak formulation of \eqref{BlowUpEqn} is 
$$
\iint_{Q_1}D\psi^k(y,\bv^{k})\cdot \phi_sdyds=\iint_{Q_1}DF^k(D\bv^{k})\cdot D\phi dyds
$$
for each $\phi\in C_c(Q_1;\R^m)$. Passing to the limit as $k=k_j\rightarrow \infty$ gives that $\bw$ is a weak solution of the linear evolution equation 
\begin{equation}\label{BlowEqnLimit}
\partial_s(D^2\psi(a)\bw)=\Div(D^2F(M)D\bw)
\end{equation}
in $Q_1$ in the sense of Definition \ref{weakSolnLoc}.

\par 3.  We claim that $\bw\in C^\infty(Q_1; \R^m)$. First observe that by Corollaries \ref{EnergyBound1cor} and \ref{EnergyBound2cor}, $\bw$ satisfies 
$$
\max_{|s|\le \frac{1}{2}}\int_{B_1}\eta^2|\bw|^2 dy+ \iint_{Q_1} \eta^2|D\bw|^2dyds 
\le C\iint_{Q_1}\left(\eta|\eta_s|+|D\eta|^2\right) |\bw|^2dyds
$$
and
$$
\max_{|s|\le \frac{1}{2}}\int_{B_1}\eta^2|D\bw|^2 dy + \iint_{Q_1}\eta^2|\bw_s|^2dyds
\le C\iint_{Q_1}\left(\eta|\eta_s|+|D\eta|^2\right) |D\bw|^2dyds
$$
for each nonnegative $\eta\in C^\infty_c(Q_1)$. Here $C=C(\lambda,\Lambda,\theta, \Theta)$.

\par Let $Q_R:=Q_R(0,0)$ for $R\in (0,1)$ and choose $h$ so small that $0<|h|<1-R\le\text{dist}(Q_R,\partial Q_1)$. We define
the time difference quotient
$$
\partial^h_s\bw(y,s):=\frac{\bw(y,s+h)-\bw(y,s)}{h}
$$
for such values of $h$ and $(y,s)\in Q_R$. Also note that $\partial^h_s\bw$ is a weak solution of the system \eqref{BlowEqnLimit} in $Q_R$. So as we observed above, 
$\partial^h_s\bw$ satisfies the energy estimates
$$
\max_{|s|\le \frac{1}{2}R^2}\int_{B_R}\eta^2|\partial^h_s\bw|^2 dy+ \iint_{Q_R} \eta^2|\partial^h_sD\bw|^2dyds 
\le C\iint_{Q_R}\left(\eta|\eta_s|+|D\eta|^2\right) |\partial^h_s\bw|^2dyds
$$
and
$$
\max_{|s|\le \frac{1}{2}R^2}\int_{B_R}\eta^2|\partial^h_sD\bw|^2 dy + \iint_{Q_R}\eta^2|\partial^h_s\bw_s|^2dyds
\le C\iint_{Q_R}\left(\eta|\eta_s|+|D\eta|^2\right) |\partial^h_sD\bw|^2dyds
$$
for $0<|h|<1-R$.  It follows that we have improved integrability of some of the derivatives of $\bw$: $\bw_{ss}\in L^2_{\text{loc}}(Q_R;\R^m)$ and $D\bw_s\in L^2_{\text{loc}}(Q_R;\Mmn)$ (Theorem 3, section 5.8.2 of \cite{Eva}).

\par We can derive the same estimates for the spatial difference quotients
$$
\partial^h_{y_i}\bw(y,s):=\frac{\bw(y+he_i,s)-\bw(y,s)}{h}
$$
for $(y,s)\in Q_R$. Here $i=1,\dots, n$ and $\{e_1\dots, e_n\}$ is the standard basis in $\R^n$.  Consequently, we can conclude $D\bw_{y_i} \in L^2_{\text{loc}}(Q_R;\Mmn)$ for each $i=1,\dots, n$.  In particular, we have that each of the second derivatives of $\bw$ are locally square integrable on $Q_1$.  Furthermore, we can proceed by induction to conclude that space-time derivatives of $\bw$ of all orders are locally square integrable on $Q_1$ which implies that $\bw\in C^\infty(Q;\R^m)$.


\par 4. A close inspection of our justification that $\bw$ is smooth leads us to conclude that we can pointwise bound the higher order space-time derivatives of $\bw$ by the integral $\iint_{Q_1}(|\bw|^2+|D\bw|^2)dyds$.  Using this fact and \eqref{WeakL22bounds}, we then conclude that there is a constant $C_1$, that depends only on $ n, \theta, \Theta, \lambda$ and $\Lambda$, such that 
\begin{align*}
\fthetaint\;\;\left|\frac{\bw-(\bw)_{Q_\vartheta}-(D\bw)_{Q_\vartheta}y}{\vartheta}\right|^2dyds +\fthetaint|D\bw-(D\bw)_{Q_{\vartheta}}|^2dyds \le C_1\vartheta^2.
\end{align*}
Here $Q_{\vartheta}:=Q_{\vartheta}(0,0)$.  Let us now choose $\vartheta_k\equiv\vartheta\in (0,1/2)$ so small that $C_1\vartheta^2\le 1/4$. In view of \eqref{CompactnessBlow}, we have for all sufficiently large $j$ 
\begin{equation}\label{LastBlowUpDisplay}
\fthetaint\;\;\left|\frac{\bv^{k_j}-(\bv^{k_j})_{Q_\vartheta}-(D\bv^{k_j})_{Q_\vartheta}y}{\vartheta}\right|^2dyds +\fthetaint|D\bv^{k_j}-(D\bv^{k_j})_{Q_{\vartheta}}|^2dyds \le \frac{3}{8}.
\end{equation}
However, it is readily verified that inequality \eqref{ContradictionMaybe} implies that the left hand side of \eqref{LastBlowUpDisplay} is larger than 1/2 for all $k\in \N$. Therefore, we have the desired contradiction.   
\end{proof}

\allowdisplaybreaks

\par We now seek to iterate the conclusion of Lemma \ref{BlowUpLemma}. First let us recall a basic fact about the integral averages of $\bv$. Observe for $\tau\in (0,1]$ and $Q_r=Q_r(x,t)\subset U\times(0,T)$, 
\begin{align}\label{IterationV}
|(\bv)_{Q_{\tau r}}- (\bv)_{Q_r}|&= \left|\intQrtau \;\left(\bv - (\bv)_{Q_r}\right)dyds\right|\nonumber\\
&= \left|\intQrtau \;\left(\bv - (\bv)_{Q_r}-(D\bv)_{Q_r}(y-x)\right)dyds\right|\nonumber\\
&\le \left(\intQrtau \; |\bv - (\bv)_{Q_r}-(D\bv)_{Q_r}(y-x)|^2dyds\right)^{1/2} \nonumber \\
&\le \frac{1}{\tau^{n/2+1}}\left(\ffint \; |\bv - (\bv)_{Q_r}-(D\bv)_{Q_r}(y-x)|^2dyds\right)^{1/2} \nonumber \\
&= \frac{r}{\tau^{n/2+1}}\left(\ffint \; \left|\frac{\bv - (\bv)_{Q_r}-(D\bv)_{Q_r}(y-x)}{r}\right|^2dyds\right)^{1/2} \nonumber \\
&\le \frac{r}{\tau^{n/2+1}}E(x,t,r)^{1/2}.
\end{align}
Similarly we have
\begin{align}\label{IterationDV}
|(D\bv)_{Q_{\tau r}}- (D\bv)_{Q_r}|&= \left|\intQrtau \;\left(D\bv - (D\bv)_{Q_r}\right)dyds\right|\nonumber\\
&\le \left(\intQrtau \; |D\bv - (D\bv)_{Q_r}|^2dyds\right)^{1/2} \nonumber \\
&\le \frac{1}{\tau^{n/2+1}}\left(\ffint \; |D\bv - (D\bv)_{Q_r}|^2dyds\right)^{1/2} \nonumber \\
&\le  \frac{1}{\tau^{n/2+1}}E(x,t,r)^{1/2}.
\end{align}

\begin{cor}\label{iterCorollary}
Assume $\bv$ is a weak solution of \eqref{mainPDE} in $U\times(0,T)$.  Let $L>0$ and select $\epsilon,\vartheta,\rho\in (0,1/2)$ as in Lemma \ref{BlowUpLemma}. If
\begin{equation}\label{IterAssump}
\begin{cases}
Q_r(x,t)\subset U\times(0,T), \quad r<\rho\\
|(\bv)_{Q_r}|, |(D\bv)_{Q_r}|< \frac{1}{2}L\\
E(x,t,r)< \epsilon_1^2
\end{cases},
\end{equation}
where $\epsilon_1:=\min\left\{\epsilon, \frac{\vartheta^{n/2+1}}{6}L\right\}$, then 
\begin{equation}\label{InterationK}
\begin{cases}
|(\bv)_{Q_{\vartheta^kr} }|, |(D\bv)_{Q_{\vartheta^k r}}|< L\\
E(x,t,\vartheta^{k}r)\le \frac{1}{2^{k}}E(x,t,r)
\end{cases}
\end{equation}
for each $k\in \N$.
\end{cor}

\begin{proof}
We will use mathematical induction on $k$.  Let us first consider the base case $k=1$. By Lemma  \ref{BlowUpLemma} and \eqref{IterAssump}, we have 
$E(x,t,\vartheta r)\le \frac{1}{2}E(x,t,r)$.  We also have from \eqref{IterationV} that
\begin{align*}
|(\bv)_{Q_{\vartheta r}}|&\le |(\bv)_{Q_{\vartheta r}}- (\bv)_{Q_r}|+|(\bv)_{Q_r}|\\
&\le \frac{r}{\vartheta^{n/2+1}}E(x,t,r)^{1/2}+|(\bv)_{Q_r}|\\
&\le \frac{\epsilon_1}{\vartheta^{n/2+1}}+\frac{1}{2}L\\
&<L.
\end{align*}
Similarly, we can employ \eqref{IterationDV} to conclude $|(D\bv)_{Q_{\vartheta r} }|< L$.  Therefore, we have established \eqref{InterationK} for $k=1$.

\par Now suppose that \eqref{InterationK} holds for  $k=1,2,\dots, j\ge 1$.  In view of \eqref{IterationV}, we can employ the triangle inequality to deduce 
\begin{align*}
|(\bv)_{Q_{\vartheta^{j+1}r} }| & \le \sum^{j-1}_{k=0}|(\bv)_{Q_{\vartheta^{k+1}r} }-(\bv)_{Q_{\vartheta^{k}r} }|+|(\bv)_{Q_{ r} }|
\\
& \le \sum^{j-1}_{k=1}\frac{r}{\vartheta^{n/2+1}}E(x,t,\vartheta^k r)^{1/2}+|(\bv)_{Q_{ r} }|\\
& < \sum^{j-1}_{k=1}\frac{1}{\vartheta^{n/2+1}}\frac{1}{\sqrt{2}^k}E(x,t,r)^{1/2}+\frac{1}{2}L\\
&\le \frac{\epsilon_1}{\vartheta^{n/2+1}}\sum^{j-1}_{k=1}\frac{1}{\sqrt{2}^k}+\frac{1}{2}L\\
&\le \frac{\epsilon_1}{\vartheta^{n/2+1}}\sum^{\infty}_{k=1}\left(\frac{3}{4}\right)^k+\frac{1}{2}L\\
&\le \frac{3\epsilon_1}{\vartheta^{n/2+1}}+\frac{1}{2}L\\
&\le L.
\end{align*}
In nearly the same fashion, we can use \eqref{IterationDV} to deduce $|(D\bv)_{Q_{\vartheta^{j+1}r} }|< L$. As 
$$
E(x,t,\vartheta^{j}r)\le \frac{1}{2^{j}}E(x,t,r)\le \frac{1}{2^j}\epsilon_1^2<\epsilon^2,
$$
Lemma \ref{BlowUpLemma} then gives 
$$
E(x,t,\vartheta^{j+1}r)=E(x,t,\vartheta(\vartheta^{j}r))\le \frac{1}{2}E(x,t,\vartheta^{j}r)\le\frac{1}{2^{j+1}}E(x,t,r).
$$
\end{proof}
\begin{cor}\label{DecayCor}
Assume $\bv$ is a weak solution of \eqref{mainPDE} in $U\times (0,T)$. Let $L>0$ and suppose there are $(x,t)\in U\times(0,T)$ and $r>0$ as in \eqref{IterAssump}.  Then there exist $C\ge 0$, $\rho_1\in (0,\rho)$, $\alpha\in (0,1)$ and a neighborhood $O\subset U\times(0,T)$ of $(x,t)$ such that 
$$
E(y,s,R)\le CR^\alpha, \quad R\in (0,\rho_1),\quad (y,s)\in O. 
$$
\end{cor}
\begin{proof}
Let $R\in (0,r)$ and choose $k\in \N$ such that $\vartheta^{k+1}r<R\le \vartheta^k r$. For $f\in L^2_{\text{loc}}((0,T); H^1_{\text{loc}}(U))$, 
we have
\begin{equation}\label{ffIneq1}
\ffRint|f-f_{Q_R}|^2dyds\le \frac{4}{\vartheta^{n+2}} \intQvarkr|f-f_{Q_{\vartheta^{k}r}}|^2dyds
\end{equation}
and 
\begin{align}\label{ffIneq2}
\ffRint\;\;\left|\frac{f-f_{Q_R}-(Df)_{Q_R}\cdot (y-x)}{R}\right|^2dyds\le \hspace{3in} \nonumber\\
\frac{4}{\vartheta^{2(n+3)}}\left\{\intQvarkr\;\;\;\;\left|\frac{f-f_{Q_{\vartheta^kr}}-(Df)_{Q_{\vartheta^kr}}\cdot (y-x) }{\vartheta^kr}\right|^2dyds+ \intQvarkr|Df-(Df)_{Q_{\vartheta^{k}r}}|^2dyds\right\}.\hspace{.5in}
\end{align}
These inequalities can be derived similarly as we did for \eqref{IterationV} and \eqref{IterationDV}; they are also proved in Corollary 4.9 of \cite{HynTAM}. 

\par Letting $f=v^i$ in \eqref{ffIneq2} and  $f=Dv^i$ in \eqref{ffIneq1} and summing over $i=1,\dots, m$ gives
\begin{align*}
E(x,t,R)&\le \frac{8}{\vartheta^{2(n+3)}}E(x,t,\vartheta^k r).
\end{align*}
In view of Corollary \ref{iterCorollary}, 
\begin{align*}
E(x,t,R)&\le \frac{8}{\vartheta^{2(n+3)}} \frac{1}{2^k}E(x,t,r)\\
& \le \frac{8\epsilon_1^2}{\vartheta^{2(n+3)}} \frac{1}{2^k}\\
&= \frac{16\epsilon_1^2}{\vartheta^{2(n+3)}} e^{-(k+1) \log 2}\\
&\le \frac{16\epsilon_1^2}{\vartheta^{2(n+3)}} \left(\frac{R}{r}\right)^\frac{\ln(1/2)}{\ln\vartheta}.
\end{align*}
\par Recall that $(\bv)_{Q_r(y,s)}$, $(D\bv)_{Q_r(y,s)}$, and $E(y,s,r)$ are all continuous functions of $(y,s)\in U\times (0,T)$ and $r>0$. Therefore, there exists $\rho_1, \rho_2\in (0,\rho)$ and a neighborhood $O$ of $(x,t)$ such that \eqref{IterAssump} holds for all $(y,s)\in O$ and $r\in(\rho_1,\rho_2)$. As a result, we can repeat the same computation above to conclude 
$$
E(y,s,R)\le \frac{16\epsilon_1^2}{\vartheta^{2(n+3)}} \left(\frac{R}{\rho_1}\right)^\frac{\ln(1/2)}{\ln\vartheta}
$$
for $(y,s)\in O$ and $R\in (0,\rho_1)$.
\end{proof}

\subsection{Partial regularity}
We are finally in position to prove Theorem \ref{mainThm} and  \ref{secondThm}. We will start with Theorem \ref{mainThm}, which asserts the almost everywhere H\"older continuity of weak solutions.

\begin{proof}[Proof of Theorem \ref{mainThm}]  We first claim that the set of points $(x,t)$ for which the following limits hold
$$
\begin{cases}
\lim_{r\rightarrow 0^+}(\bv)_{Q_r(x,t)}=\bv(x,t)\\
\lim_{r\rightarrow 0^+}(D\bv)_{Q_r(x,t)}=D\bv(x,t)\\
\lim_{r\rightarrow 0^+}E(x,t,r)=0
\end{cases}
$$
has full Lebesgue measure in $U\times(0,T)$.  This is evident for the first two limits by a version of Lebesgue's differentiation theorem for parabolic cylinders \cite{ImbSyl}. 

\par As for the third limit, recall Poincar\'{e}'s
inequality on a cylinder $Q_r=Q_r(x,t)\subset U\times(0,T)$: there is a constant $C_0$ such that
$$
\iint_{Q_r}|\bw - (\bw)_{Q_r}|^2dyds \le C_0\left\{r^4\iint_{Q_r}|\bw_t|^2dyds + r^2\iint_{Q_r}|D\bw|^2dyds\right\}
$$
for each $\bw\in H^1_{\text{loc}}(U\times(0,T);\R^m)$. Substituting 
$$
\bw(y,s):=\bv(y,s) - (\bv)_{Q_r}-(D\bv)_{Q_r}(y-x)
$$
in the Poincar\'e inequality above gives
\begin{equation}\label{SimpleUpperBoundE}
E(x,t,r)\le (C_0+1)\left\{ 
r^2\ffint\;\; \left|\bv_t\right|^2dyds + \ffint\;\; \left|D\bv-(D\bv)_{Q_r}\right|^2dyds\right\}.
\end{equation}
Again we invoke Lebesgue differentiation \cite{ImbSyl} to conclude $\lim_{r\rightarrow 0^+}E(x,t,r)=0$ on a set of full Lebesgue measure.

\par It now follows that for Lebesgue almost every $(x,t)\in U\times(0,T)$ there are $L,r>0$ such that \eqref{IterAssump} holds.  At any such $(x,t)\in U\times(0,T)$, we can then apply the conclusion of Corollary \ref{DecayCor}.  In particular, there is a neighborhood $O\subset U\times(0,T)$ and $\rho_1>0$ such that $E(y,s,R)\le CR^\alpha$ for $Q_R=Q_R(y,s)\subset U\times(0,T)$ with $(y,s)\in O$ and $R<\rho_1$.   Therefore,
$$
\left(\ffRint|D\bv-(D\bv)_{Q_R}|^2dyds\right)^{1/2}\le E(y,s,R)^{1/2}\le \sqrt{C} R^{\frac{1}{2}\alpha}.
$$
By Campanato's criterion \cite{Camp, DaPrato}, $D\bv$ is H\"older continuous in a neighborhood of $(x,t)$. Consequently, the set ${\cal O}$ of points $(x,t)$ for which there is a neighborhood of $(x,t)$ such that $D\bv$ is H\"{o}lder continuous has full Lebesgue measure. By definition, ${\cal O}$ is open.  This concludes a proof of Theorem \ref{mainThm}.
\end{proof}
\begin{rem}
The above argument also shows that $\bv$ is H\"older continuous in a neighborhood of $(x,t)$. Indeed
\begin{align*}
\left(\ffRint\;\; \left|\bv-(\bv)_{Q_R}\right|^2dzd\tau\right)^{1/2}&\le \left(\ffRint\;\; \left|\bv-(\bv)_{Q_R}-(D\bv)_{Q_R}(z-y)\right|^2dzd\tau\right)^{1/2}+|(D\bv)_{Q_R}|R\\
&= RE(y,s,R)^{1/2}+LR\\
&\le \sqrt{C} R^{1+\frac{1}{2}\alpha}+LR\\
&\le (\sqrt{C} \rho_1^{\frac{1}{2}\alpha}+L)R.
\end{align*}
\end{rem}
Let us recall the definition of parabolic Hausdorff measure.  

\begin{defn}\label{ParaHausMeas}
For $G\subset \R^n\times\R$, $s\in [0,n+2]$, $\delta>0$, set 
$$
{\cal P}^s_\delta(G):=\inf\left\{\sum_{i\in \N}r^s_i: G\subset\bigcup_{i\in \N} Q_{r_i}(x_i,t_i),\; r_i\le \delta\right\}.
$$
The {\it $s$-dimensional parabolic Hausdorff measure} of $G$ is defined 
\begin{equation}
{\cal P}^s(G):=\sup_{\delta>0}{\cal P}^s_\delta(G).
\end{equation}
Moreover, the {\it parabolic Hausdorff dimension} of $G$ is the number
$$
\text{dim}_{\cal P}(G):=\inf\{s\ge 0:{\cal P}^s(G)=0 \}.
$$
\end{defn}
We note that ${\cal P}^s$ is an outer measure on $\R^n\times\R$ for each $s\in [0,n+2]$. While there are many important properties of (general) Hausdorff measure (as detailed in \cite{Evans} and \cite{Rogers}),  we will only make use of one fact about parabolic Hausdorff measure that is based on a Poincar\'e inequality for fractionally differentiable functions. The Poincar\'e inequality we have in mind is as follows, we will omit a proof of this inequality as it is stated and proved in Lemma 2.16 of \cite{DuzMinSte}. 
\begin{lem}
Let $\gamma\in (0,1)$.  Suppse $w\in L^2_{\text{loc}}(U\times(0,T))$ 
satisfies 
\begin{equation}\label{FracSpaceW}
\int^{t_1}_{t_0}\int_{V}\int_V\frac{|w(x,t)-w(y,t)|^2}{|x-y|^{n+2\gamma}}dxdydt+\int^{t_1}_{t_0}\int^{t_1}_{t_0}\int_V\frac{|w(x,t)-w(x,s)|^2}{|t-s|^{1+\gamma}}dxdtds<\infty
\end{equation}
for each open $V\subset\subset U$ and interval $[t_0,t_1]\subset (0,T)$.  Then there is a constant $C_0$ such that for every $Q_r=Q_r(x,t)\subset U\times(0,T)$, 
\begin{align}
&\iint_{Q_r}|w-(w)_{Q_r}|^2dyds\le C_0r^{2\gamma}\left\{\int^{t+r^2/2}_{t-r^2/2}\int_{B_r(x)}\int_{B_r(x)}\frac{|w(x,s)-w(y,s)|^2}{|x-y|^{n+2\gamma}}dxdyds\right.\\
&\hspace{2in} \left.+\int^{t+r^2/2}_{t-r^2/2}\int^{t+r^2/2}_{t-r^2/2}\int_{B_r(x)}\frac{|w(y,\tau)-w(y,s)|^2}{|\tau-s|^{1+\gamma}}dyd\tau ds\right\}.
\end{align}
\end{lem}
The crucial fact about parabolic Hausdorff measure is as follows.  Versions of this assertion which can be found in Proposition 3.3 of \cite{DuzMin}, Proposition 4.2 in \cite{Min}, Theorem 3 in section 2.4.3 of \cite{Evans}, so we will not provide a proof.

\begin{prop}\label{DimEstProp}
Let $\gamma\in (0,1)$, and suppose $w\in L^2_{\text{loc}}(U\times(0,T))$ satisfies \eqref{FracSpaceW}.  Then 
$$
\text{\normalfont{dim}}_{\cal P}\left(\left\{(x,t)\in U\times (0,T):  \limsup_{r\rightarrow 0^+}\ffintXT|w-(w)_{Q_r(x,t)}|^2dyds>0\right\}\right)\le n+2-2\gamma
$$
and
$$
\text{\normalfont{dim}}_{\cal P}\left(\left\{(x,t)\in U\times (0,T):  \limsup_{r\rightarrow 0^+}|(w)_{Q_r(x,t)}|=+\infty\right\}\right) \le n+2-2\gamma.
$$
\end{prop}
We will combine these facts with our previous estimates to fashion a fairly simple proof of Theorem \ref{secondThm}. Before proceeding to the proof, we will need to verify a local version of inequality \eqref{EnergyBound2}.
\begin{lem}
Assume $\bv$ is a weak solution of \eqref{mainPDE} on $U\times(0,T)$. There is a constant $C$ depending only on $\theta,\lambda, \Theta$, and $\Lambda$ such that 
\begin{equation}\label{MyCacci}
\iint_{Q_r(x,t)}|\bv_t|^2dyds\le \frac{C}{r^2}\iint_{Q_{2r}(x,t)}|D\bv-(D\bv)_{Q_{2r}(x,t)}|^2dyds
\end{equation}
whenever $Q_{2r}(x,t)\subset U\times(0,T)$. 
\end{lem}

\begin{proof}
Fix $A\in \Mmn$ and $\phi\in C_c^\infty(U\times(0,T))$.  As we computed \eqref{MainIdentity2}, we find
\begin{align*}
&\frac{d}{dt}\int_U\phi \left(F(D\bv)-DF(A)\cdot(DV-A)\right)dx+\int_U\phi \partial_t(D\psi(\bv))\cdot \bv_t dx \\
&\hspace{.5in}=\int_U\left(\phi_t \left(F(D\bv)-DF(A)\cdot(DV-A)\right) - \bv_t \cdot (DF(D\bv)-DF(A))D\phi \right)dx
\end{align*} 
for almost every $t\in (0,T)$. And setting $\phi=\eta^2$ for $\eta\ge 0$ gives
\begin{equation}\label{MyCacciOnTheWay}
\max_{0\le t\le T}\int_U\eta^2|D\bv-A|^2 dx + \int^T_0\int_U \eta^2|\bv_t|^2dxdt 
\le C\int^T_0\int_U\left(\eta|\eta_t|+|D\eta|^2\right) |D\bv-A|^2dxdt
\end{equation}
the same way that we derived \eqref{EnergyBound2}.

\par Now let $\eta_0\in C^\infty_c(\R^n)$ satisfy
$$
\begin{cases}
0\le \eta_0\le 1\\
\eta_0\equiv 1\;\;\text{in}\;\; B_r(x)\\
\eta_0\equiv 0\;\;\text{in}\;\; \R^n\setminus B_{2r}(x)\\
|D\eta_0|\le 2/r
\end{cases}
$$
and $\eta_1\in C^\infty_c(\R)$ satisfy
$$
\begin{cases}
0\le \eta_1\le 1\\
\eta_1\equiv 1\;\;\text{in}\;\; (t-r^2/2,t+r^2/2)\\
\eta_1\equiv 0\;\;\text{in}\;\; \R\setminus (t-2r^2,t+2r^2)\\
|\partial_\tau\eta_1|\le 2/r^2.
\end{cases}
$$
We conclude \eqref{MyCacci} by choosing  $\eta(y,\tau) = \eta_0(y)\cdot \eta_1(\tau)$ and $A=(D\bv)_{Q_{2r}(x,t)}$ in \eqref{MyCacciOnTheWay}.
\end{proof}

\begin{proof}[Proof of Theorem \ref{secondThm}]  
Choose $\beta$ as in \eqref{BetaInterval} and select $\epsilon>0$ so small that
\begin{equation}\label{BetaInterval2}
\beta+\epsilon\in \left(0,\frac{1}{2}-\frac{1}{p}\right).
\end{equation}
Let us also recall our definition of ${\cal O}$
$$
{\cal O}=\{(x,t)\in U\times (0,T): D\bv\; \text{is H\"older continuous in some neighborhood of $(x,t)$}\}.
$$
By Corollary \ref{DecayCor}, 
$$
U\times(0,T)\setminus{\cal O}\subset G_1 \cup G_2\cup G_3.
$$
Here
$$
G_1=\left\{(x,t)\in U\times (0,T):  \limsup_{r\rightarrow 0^+}E(x,t,r)>0\right\},
$$
$$
G_2=\left\{(x,t)\in U\times (0,T):  \limsup_{r\rightarrow 0^+}|(\bv)_{Q_r(x,t)}|=+\infty\right\},
$$
and 
$$
G_3=\left\{(x,t)\in U\times (0,T):  \limsup_{r\rightarrow 0^+}|(D\bv)_{Q_r(x,t)}|=+\infty\right\}.
$$
It suffices to show ${\cal P}^{n+2-2\beta}(G_i)=0$ for $i=1,2,3.$
\par Observe by \eqref{SimpleUpperBoundE} and 
\eqref{MyCacci}
$$
\limsup_{r\rightarrow 0^+}E(x,t,r)\le 2(C_0+1) \limsup_{r\rightarrow 0^+}\ffintXT|D\bv-(D\bv)_{Q_r(x,t)}|^2dyds
$$
for any $(x,t)\in U\times (0,T)$.  By Lemma \ref{H2prop} and Corollary \ref{TimeDiffDVCor}, $w=v^i_{x_j}$ satisfies \eqref{FracSpaceW} with $\gamma=\beta+\epsilon$ for each $i=1,\dots, m$ and $j=1,\dots, n$.  Here 
we are using the inclusion $H^1_{\text{loc}}(U)\subset H^\sigma_{\text{loc}}(U)$ $(0<\sigma<1)$, which is proved in Proposition 2.2 of \cite{DiHitch}. Therefore,
$$
G_1\subset \left\{(x,t)\in U\times (0,T):  \limsup_{r\rightarrow 0^+} \ffintXT|D\bv-(D\bv)_{Q_r(x,t)}|^2dyds>0\right\}.
$$
In view of Proposition \ref{DimEstProp}, we conclude 
$$
\text{dim}_{\cal P}(G_1)\le n+2-2(\beta+\epsilon)<n+2-2\beta.
$$
It follows that ${\cal P}^{n+2-2\beta}(G_1)=0$. Likewise, we can also conclude ${\cal P}^{n+2-2\beta}(G_3)=0$.  The conclusion ${\cal P}^{n+2-2\beta}(G_2)=0$ follows similarly as $v^i$ satisfies \eqref{FracSpaceW} for every $\gamma\in (0,1)$ and $i=1,\dots, m$.
\end{proof}

\appendix
\section{Dirichlet problem}\label{ExistenceApp} 
Assume $\psi\in C^2(\R^m) $ and $F\in C^2(\Mmn)$ satisfy \eqref{UnifConv}, \eqref{UnifConv2} and \eqref{ZeroDeriv}. For a given $\bg\in H^1(U;\R^m)$, we will show that a weak solution of the initial value problem 
\begin{equation}\label{DirichletIVP}
\begin{cases}
\partial_t\left(D\psi(\bv)\right)=\Div DF(D\bv), \quad & \text{in}\;U\times(0,T)\\
\hspace{.67in}\bv =0, \quad &\text{on}\;\partial U\times[0,T)\\
\hspace{.67in}\bv =\bg,\quad  \quad &\text{on}\; U\times\{0\}
\end{cases}
\end{equation} 
exists.  In particular, the solution $\bv$ we construct solution will also be a weak solution of \eqref{mainPDE}. We also acknowledge that the existence of a solution to \eqref{DirichletIVP} has already been established in various contexts such as \cite{AltLuck, DiaThe, Vis}. We have added this appendix because we will require more integrability of our weak solutions (see \eqref{naturalbounds2} below). 

\par Note that any smooth solution $\bv$ of \eqref{DirichletIVP} satisfies 
$$
\frac{d}{dt}\int_U\psi^*(D\psi(\bv))dx =- \int_UDF(D\bv)\cdot D\bv dx
$$
and therefore 
\begin{equation}\label{DirIdentity1}
\int_U\psi^*(D\psi(\bv(x,t)))dx +\int^t_0\int_UDF(D\bv)\cdot D\bv dx ds = \int_U\psi^*(D\psi(\bg))dx 
\end{equation}
for each $t\ge 0$. It follows that 
\begin{equation}\label{GlobalEnergyBound}
\sup_{0\le t\le T}\int_U|\bv(x,t)|^2dx+\int^T_0\int_U|D\bv|^2dxdt \le C\int_U|\bg|^2dx
\end{equation}
for a constant $C$ depending only on $\theta,\lambda,\Theta$ and $\Lambda$. 

\par The identity
$$
\frac{d}{dt}\int_UF(D\bv)dx=-\int_U D^2\psi(\bv)\bv_t\cdot\bv_t dx
$$
also holds for any smooth solution $\bv$ of \eqref{DirichletIVP}. Integrating this identity gives
\begin{equation}\label{DirIdentity2}
\int_UF(D(\bv(x,t))dx +\int^t_0\int_U D^2\psi(\bv)\bv_t\cdot\bv_t dxds = \int_UF(D\bg)dx 
\end{equation}
for each $t\ge 0$. Using standard manipulations, we can also deduce 
\begin{equation}\label{GlobalEnergy2}
\sup_{0\le t\le T}\int_U|D\bv(x,t)|^2dx+\int^T_0\int_U|\bv_t|^2dxds\le C\int_U|D\bg|^2dx
\end{equation}
for some constant $C$ only depending on $\theta,\Theta,\lambda$ and $\Lambda$.

\par Inequalities \eqref{GlobalEnergyBound} and \eqref{GlobalEnergy2} leads us to define weak solutions analogously to Definition \ref{weakSolnLoc}.

\begin{defn}\label{WeakSolnIVPDef} Suppose $\bg\in H^1(U;\R^m)$.  A measurable mapping $\bv:U\times[0,T]\rightarrow \R^m$ is a {\it weak solution} of \eqref{DirichletIVP} if $\bv$ satisfies: $(i)$
\begin{equation}\label{naturalbounds2}
\bv\in L^\infty([0,T]; H^1_0(U;\R^m))\quad \text{and}\quad \bv_t\in L^2(U\times [0,T];\R^m);
\end{equation}
$(ii)$
\begin{equation}\label{weakSolnCond2}
\int^T_0\int_U D\psi(\bv)\cdot \bw_t dxdt=\int^T_0\int_U DF(D\bv)\cdot D\bw dxdt,
\end{equation}
for all $\bw\in C^\infty_c(U\times(0,T); \R^m)$; and $(iii)$
\begin{equation}\label{InitialCond}
\bv(\cdot, 0)=\bg.
\end{equation}
\end{defn}

\par As we argued for weak solutions of \eqref{mainPDE}, we can deduce that \eqref{DirIdentity1} and \eqref{DirIdentity2} hold for weak solutions of \eqref{DirichletIVP}. It follows of course that  \eqref{GlobalEnergyBound} and \eqref{GlobalEnergy2} also are valid for weak solutions of \eqref{DirichletIVP}.  Moreover, we can verify 
\begin{equation}\label{DVCont}
\bv\in C([0,T]; H^1(U;\R^m))
\end{equation}
similar to how we argued in part 4 of Proposition \ref{IdentityLemma}. Therefore, it makes sense to prescribe the initial condition \eqref{InitialCond}.

\par One way of constructing a weak solution to \eqref{DirichletIVP} is to use the following implicit time scheme. Let $N\in \N$, set $\tau=T/N$ and $\bv^0=\bg$, then find $\bv^k\in H^1_0(U;\R^m)$ that satisfies
\begin{equation}\label{ITSDirichletIVP}
\begin{cases}
\displaystyle\frac{D\psi(\bv^k)-D\psi(\bv^{k-1})}{\tau}=\Div DF(D\bv^k), \quad & \text{in}\; U\\
\hspace{1.32in}\bv^k=0, \quad & \text{on}\;\partial U
\end{cases}
\end{equation}
weakly for $k=1,\dots, N$. That is, 
\begin{equation}\label{weakDirichletIVP}
\int_U \frac{D\psi(\bv^k)-D\psi(\bv^{k-1})}{\tau}\cdot \bw dx+\int_U DF(D\bv^k)\cdot D\bw dx=0
\end{equation}
for each $\bw \in H^1_0(U;\R^m)$. In particular, once $\bv^1,\dots, \bv^j$ are found, then $\bv^{j+1}$ is the unique minimizer of the strictly convex and coercive functional 
$$
H^1_0(U;\R^m)\ni \bu \mapsto \int_{U}\left(F(D\bu)+\frac{\psi(\bu)-D\psi(\bv^j)\cdot \bu}{\tau}\right)dx.
$$
Therefore, \eqref{DirichletIVP} has a unique solution $\{\bv^1,\dots, \bv^N\}\subset H^1_0(U;\R^m)$. 

\par Let us now derive a few estimates on solutions of the implicit time scheme analogous to \eqref{GlobalEnergyBound} and \eqref{GlobalEnergy2}.

\begin{lem}
Suppose $\{\bv^1,\dots, \bv^N\}\subset H^1_0(U;\R^m)$ is a solution of the implicit time scheme \eqref{ITSDirichletIVP} with $\bv^0=\bg$. Then there is a constant $C$ such that
\begin{equation}\label{GlobalBounddiscrete}
\max_{1\le k\le N}\int_U|\bv^k|^2dx+\tau\sum^N_{k=1}\int_U|D\bv^k|^2dx \le C\int_U|\bg|^2dx
\end{equation}
and
\begin{equation}\label{GlobalEnergy2discrete}
\max_{1\le k\le N}\int_U|D\bv^k|^2dx+\tau\sum^N_{k=1}\int_U\left|\frac{\bv^k-\bv^{k-1}}{\tau}\right|^2dx\le C\int_U|D\bg|^2dx.
\end{equation}
\end{lem}
\begin{proof}
 Choosing $\bw=\bv^k$ in \eqref{weakDirichletIVP} gives

\begin{align*}
\int_U DF(D\bv^k)\cdot D\bv^k dx&=\int_U \frac{D\psi(\bv^{k-1})-D\psi(\bv^{k})}{\tau}\cdot \bv^k dx\\
&=\int_U \frac{D\psi^*(D\psi(\bv^k)) \cdot\left(D\psi(\bv^{k-1})-D\psi(\bv^{k}\right)}{\tau}dx\\
&\le \int_U\frac{\left( \psi^*(D\psi(\bv^{k-1})) -\psi^*(D\psi(\bv^k)) \right)}{\tau}dx.
\end{align*}
Summing over $k=1,\dots, j\le N$ we find
$$
\int_U\psi^*(D\psi(\bv^{j}))dx+\tau\sum^j_{k=1}\int_U DF(D\bv^k)\cdot D\bv^k dx\le \int_U\psi^*(D\psi(\bg))dx.
$$
Consequently, we can now employ elementary manipulations to derive \eqref{GlobalBounddiscrete}. 

\par Selecting 
$\bw=\bv^k-\bv^{k-1}$ in \eqref{weakDirichletIVP} gives
\begin{align*}
\int_U\frac{D\psi(\bv^k)-D\psi(\bv^{k-1})}{\tau}\cdot \frac{\bv^k-\bv^{k-1}}{\tau}dx &= 
-\int_U DF(D\bv^k)\cdot  \frac{D\bv^k-D\bv^{k-1}}{\tau}dx\\
&\le \int_U \frac{F(D\bv^{k-1})-F(D\bv^{k})}{\tau}dx.
\end{align*}
Summing over $k=1,\dots, j\le N$ gives  
$$
\int_UF(D\bv^{j})dx+\tau\sum^j_{k=1}\int_U\frac{D\psi(\bv^k)-D\psi(\bv^{k-1})}{\tau}\cdot \frac{\bv^k-\bv^{k-1}}{\tau}dx\le \int_UF(D\bg)dx.
$$
It is now routine to conclude \eqref{GlobalEnergy2discrete}.
\end{proof}

\par We will now use \eqref{GlobalBounddiscrete} and \eqref{GlobalEnergy2discrete} to show that  \eqref{DirichletIVP} has a weak solution. Consequently, there is at least one weak solution of \eqref{mainPDE} in $U\times (0,T)$ that is partially regular as described in Theorems \ref{mainThm} and \ref{secondThm}. 
\begin{prop}
There exists a weak solution $\bv$ of \eqref{DirichletIVP}.
\end{prop}
\begin{proof} Let $N\in \N, \tau=T/N$, and suppose $\{\bv^1,\dots, \bv^N\}\subset H^1_0(U;\R^m)$ is a solution of the implicit time scheme \eqref{ITSDirichletIVP} with $\bv^0=\bg\in H^1(U;\R^m)$. Set $\tau_k:=\tau k$ for $k=0,\dots,N$ and define
$$
\begin{cases}
\bv^N(\cdot,t):=
\begin{cases}
\bg, \quad &t=0\\
\bv^k, \quad &\tau_{k-1}< t\le \tau_k
\end{cases}
\\\\
\bw^N(\cdot,t):=D\psi(\bv^{k-1})+\displaystyle\left(\frac{t-\tau_{k-1}}{\tau}\right)\left(D\psi(\bv^{k})-D\psi(\bv^{k-1})\right), \quad \tau_{k-1}\le t\le \tau_k
\end{cases}
$$
for $t\in [0,T]$.  It is straightforward to use \eqref{ITSDirichletIVP} and check that
$$
\|\partial_t(\bw^N)\|_{H^{-1}(U;\R^m)}=\|DF(D\bv^N)\|_{L^2(U;\Mmn)}
$$
holds for $t\in (0,T)\setminus\{\tau_1,\dots,\tau_N\}$. 

\par Furthermore, it is routine to use \eqref{GlobalBounddiscrete} to show  
$$
\sup_{0\le t\le T}\int_U|\bv^N(x,t)|^2dx+\int^T_0\int_U|D\bv^N|^2dxdt \le C\int_U|\bg|^2dx,
$$
which in turn implies that $\bw^N$ fulfills
$$
\sup_{0\le t\le T}\int_U|\bw^N(x,t)|^2dx+\int^T_0\|\partial_t(\bw^N(\cdot, t))\|_{H^{-1}(U;\R^m)}^2dt \le C\int_U|\bg|^2dx.
$$
By the proof given in Section III.1 of \cite{Vis}, there is $\bv$ satisfying 
$$
\bv\in C([0,T]; L^2(U;\R^m))\cap  L^2([0,T]; H^1_0(U;\R^m))
$$
and subsequences $(\bv^{N_j})_{j\in \N}$ and $(\bw^{N_j})_{j\in \N}$ such that $\bv^{N_j}\rightarrow \bv$ in $L^2([0,T]; H^1_0(U;\R^m))$ and  $\partial_t(\bw^{N_j})\rightarrow \partial_t(D\psi(\bv))$ in $L^2([0,T]; H^{-1}(U;\R^m))$.  Moreover, $\bv$ satisfies the weak solution condition \eqref{weakSolnCond2}. Therefore, we are only left to verify that $\bv$ also satisfies \eqref{naturalbounds2}.

\par To this end, we define 
$$
\bu^N(\cdot,t):=\bv^{k-1}+\displaystyle\left(\frac{t-\tau_{k-1}}{\tau}\right)\left(\bv^{k}-\bv^{k-1}\right), \quad \tau_{k-1}\le t\le \tau_k
$$
for $t\in [0,T]$.  By \eqref{GlobalEnergy2discrete},  
$$
\sup_{0\le t\le T}\int_U|D\bu^N(x,t)|^2dx+\int^T_0\int_U|\partial_t\bu^N|^2dxds\le C\int_U|D\bg|^2dx.
$$
It follows that $(\bu^N(\cdot,t))_{N\in\N}\subset H^1_0(U;\R^m)$ is bounded for each $t\in [0,T]$ and that $\bu^N: [0,T]\rightarrow L^2(U;\R^m)$ is uniformly equicontinuous.  Since $H^1_0(U;\R^m)\subset L^2(U;\R^m)$ with compact embedding, there we can pass to the limit along a subsequence $(\bu^{N_j})_{j\in \N}$ (that will not be labeled) to find a $\bu$ such that $\bu^{N_j}\rightarrow \bu$ in $C([0,T]; L^2(U;\R^m))$ (by Aubin's compactness theorem \cite{Aubin}) and $\partial_t\bu^{N_j}\rightharpoonup \partial_t\bu$ in $L^2(U\times (0,T);\R^m)$.  In particular, we note that $\bu$ satisfies  
\eqref{naturalbounds2}.

\par In order to conclude, it suffices to show $\bv\equiv\bu$. To see this, we recall \eqref{GlobalEnergy2discrete} and compute
\begin{align*}
\int^T_0\int_U|\bu^N-\bv^N|^2dxdt&=\sum^N_{k=1}\int^{\tau_k}_{\tau_{k-1}}\int_U|\bv^N-\bu^N|^2dxdt\\
&=\sum^N_{k=1}\int^{\tau_k}_{\tau_{k-1}}\int_U\left|\left(1-\left(\frac{t-\tau_{k-1}}{\tau}\right)\right)(\bv^k-\bv^{k-1})\right|^2dxdt\\
&=\sum^N_{k=1}\int^{\tau_k}_{\tau_{k-1}}\left(1-\left(\frac{t-\tau_{k-1}}{\tau}\right)\right)^2dt\int_U\left|\bv^k-\bv^{k-1}\right|^2dx\\
&=\sum^N_{k=1}\int^{\tau_k}_{\tau_{k-1}}\left(\frac{\tau_{k}-t}{\tau}\right)^2dt\int_U\left|\bv^k-\bv^{k-1}\right|^2dx\\
&=\sum^N_{k=1}\frac{\tau}{3}\int_U\left|\bv^k-\bv^{k-1}\right|^2dx\\
&=\frac{\tau^2}{3}\sum^N_{k=1}\int_U\frac{\left|\bv^k-\bv^{k-1}\right|^2}{\tau}dx\\
&\le \frac{\tau^2}{3}\left(C\int_U|D\bg|^2dx\right)\\
&\le \frac{T^2}{3N^2}\left(C\int_U|D\bg|^2dx\right).
\end{align*}
Letting $N=N_j$ and sending $j\rightarrow\infty$ gives 
$$
\int^T_0\int_U|\bu-\bv|^2dxdt=0.
$$
Consequently, $\bv$ satisfies \eqref{naturalbounds2} and is therefore a weak solution of \eqref{DirichletIVP}.
\end{proof}



\begin{thebibliography}{}

\bibitem{AltLuck} Alt, H.; Luckhaus, S. \emph{Quasilinear elliptic-parabolic differential equations.} 
Math. Z. 183 (1983), no. 3, 311--341. 

\bibitem{AGS} Ambrosio, L.; Gigli, N.; SavarŽ\'{e}, G. \emph{Gradient flows in metric spaces and in the space of probability measures}. Second edition. Lectures in Mathematics ETH ZŸrich. BirkhŠuser Verlag, Basel, 2008.

\bibitem{Aubin} Aubin, J.-P.  \emph{Un th\'eor\`eme de compacit\'e.} C. R. Acad. Sci. Paris 256 (1963) 5042--5044. 

\bibitem{BDMI} B\"ogelein, V.; Duzaar, F.; Mingione, G. \emph{The boundary regularity of non-linear parabolic systems. I.}
Ann. Inst. H. Poincar\'e Anal. Non Lin\'eaire 27 (2010), no. 1, 201--255. 

\bibitem{BDMII} B\"ogelein, V.; Duzaar, F.; Mingione, G. \emph{The boundary regularity of non-linear parabolic systems. II.}
Ann. Inst. H. Poincar\'e Anal. Non Lin\'eaire 27 (2010), no. 1, 145--200. 


\bibitem{Camp} Campanato, S. \emph{Propriet\`{a} di h\"{o}lderianit\`{a} di alcune classi di funzioni}. Ann. Scuola Norm. Sup. Pisa (3) 17 1963 175--188.

\bibitem{Campanato1} Campanato, S. \emph{On the nonlinear parabolic systems in divergence form. H\"older continuity and partial H\"older continuity of the solutions.} Ann. Mat. Pura Appl. 137 (4) (1984) 83--122.

\bibitem{Chalmers} Chalmers, B. \emph{Principles of Solidification}. Wiley, New York 1964.

\bibitem{Constantin}  Constantin, P.; Cordoba, D.; Gancedo, F.; Strain, R. \emph{On the global existence for the Muskat problem}. J. Eur. Math. Soc. 15, (2013) 201--227.


\bibitem{DaPrato} Da Prato, G. \emph{Spazi ${\cal L}^{p,\theta}(\Omega,\delta)$ e loro propriet\`{a}}.  Ann. Mat. Pura Appl. (4) 69 1965 383--392. 

\bibitem{DeG} De Giorgi, E. \emph{ Un esempio di estremali discontinue per un problema variazionale di tipo ellittico.} Boll. Un. Mat. Ital. (4) 1 1968 135--137.

\bibitem{DiHitch} Di Nezza, E.; Palatucci, G.; Valdinoci, E. \emph{Hitchhiker's guide to the fractional Sobolev spaces}. Bull. Sci. Math. 136 (2012), no. 5, 521--573.

\bibitem{DiaThe} Diaz, J.; de Th\'elin, F. \emph{On a nonlinear parabolic problem arising in some models related to turbulent flows}. SIAM J. Math. Anal. 25 (1994), no. 4, 1085--1111.

\bibitem{DuzMin} Duzaar, F.; Mingione, G. \emph{Second order parabolic systems, optimal regularity, and singular sets of solutions}. Ann. I. H. Poincar\'{e} AN 22 (2005) 705--751.

\bibitem{DuzMinSte} Duzaar, F.; Mingione, G.; Steffen, K. \emph{Parabolic systems with polynomial growth and regularity.}  Mem. Amer. Math. Soc. 214 (2011), no. 1005.


\bibitem{Eva} Evans, L. C. \emph{Partial differential equations}. Second edition. Graduate Studies in Mathematics, 19. American Mathematical Society, Providence, RI, 2010.

\bibitem{Evans} Evans, L. C.; Gariepy, R . \emph{Measure theory and fine properties of functions}. Studies in Advanced Mathematics. CRC Press, Boca Raton, FL, 1992.

\bibitem{Friedman} Friedman, A. \emph{The Stefan problem in several space variables}. Trans. Amer. Math. Soc. 133 1968 51--87.

\bibitem{Friedman2} Friedman, A. Correction to:  \emph{The Stefan problem in several space variables}. Trans. Amer. Math. Soc. 142 1969 557.

\bibitem{GiaMar} Giaquinta, M.; Martinazzi, L.. \emph{An introduction to the regularity theory for elliptic systems, harmonic maps and minimal graphs}. Second edition. Edizioni della Normale, Pisa (2012).


\bibitem{Guisti} Giusti, E.; Miranda, M. \emph{Un esempio di soluzioni discontinue per un problema di minimo relativo ad un integrale regolare del calcolo delle variazioni}. Boll. UMI 2 (1968), 1--8.

\bibitem{Gurtin} Gurtin, M. \emph{On a theory of phase transitions with interfacial energy}. 
Arch. Rational Mech. Anal. 87 (1985), no. 3, 187--212.

\bibitem{Gurtin2} Gurtin, M. \emph{Multiphase thermo-mechanics with interfacial structure I. Heat conduction and the capillary balance law}. Arch. Rational Mech. Anal. 104 (1988) 195--221.

\bibitem{Hyn} Hynd, R. \emph{Partial regularity for type two doubly nonlinear parabolic systems}. arxiv.org/abs/1704.05602.

\bibitem{HynTAM} Hynd, R. \emph{Compactness methods for doubly nonlinear parabolic systems}. Transactions of the American Mathematical Society 369 (2017), no. 7, 5031--5068.

\bibitem{ImbSyl}  Imbert C.; Silvestre, L. \emph{An introduction to fully nonlinear parabolic equations, An introduction to the K\"ahler-Ricci flow}.	Lecture Notes in Math., vol. 2086, Springer, Cham, 2013, pp. 7--88.

\bibitem{Ivanov} Ivanov, A. \emph{Regularity for doubly nonlinear parabolic equations}. J. Math. Sci. 83 (1) 1997  22--37. 
 
\bibitem{Kuusi} Kuusi, T.; Laleoglu, R. Siljander, J.; Urbano, J. \emph{H\"older continuity for Trudinger's equation in measure spaces.}
Calc. Var. Partial Differential Equations 45 (2012), no. 1-2, 193 -- 229.

\bibitem{Lawson} Lawson, H.B.; Osserman, R. \emph{Non-existence, non-uniqueness and irregularity of solutions to the minimal surface system.} Acta math. 139 (1977), 1--17.

\bibitem{Min} Mingione, G. \emph{The singular set of solutions to non-differentiable elliptic systems}. Arch. Ration. Mech. Anal. 166 (2003), no. 4, 287--301.

\bibitem{Oleinik} Ole\u{i}nik, O. A. \emph{A method of solution of the general Stefan problem.}
Soviet Math. Dokl. 1 1960 1350--1354. 

\bibitem{Perez} Perez, M. \emph{Gibbs--Thomson effects in phase transformations.} Scripta Materialia 52 (2005) 709--712.

\bibitem{Porzio} Porzio, M.; Vespri, V. \emph{H\"older estimates for local solutions of some doubly nonlinear degenerate parabolic equations.} 
J. Differential Equations 103 (1993), no. 1, 146--178.

\bibitem{Richardson} Richardson, S. \emph{Hele-Shaw flow with a free boundary produced by the injections of a fluid into a narrow channel}. J. Fluid. Mech. 56 (1972) 609--618.

\bibitem{Richardson2} Richardson, S. \emph{Some Hele-Shaw flows with time-dependent free boundaries}. J. Fluid Mech. 102 (1981), 263--278.

\bibitem{Rogers} Rogers, C. A. \emph{Hausdorff Measures}. Cambridge University Press, Cambridge, 1970.

\bibitem{Rubin} Rubinstein, L. \emph{On the determination of the position of the boundary which separates two phases in the one-dimensional problem of Stefan.} Dokl. Acad. Nauk USSR 58 (1947) 217--220.

\bibitem{Saff} Saffman, P. G.; Taylor, G. \emph{The penetration of a fluid into a porous medium or Hele-Shaw cell containing a more viscous liquid}. Proc. Roy. Soc. London. Ser. A 245 1958 312--329. 

\bibitem{Simon} Simon, J. \emph{Compact sets in the space $L^p(0,T;B)$.} Ann. Mat. Pura Appl. (4) 146 (1987), 65--96. 

\bibitem{Stefan} Stefan, J. \emph{\"Uber einige Probleme der Theorie der W\"armeleitung}. Sitzungber., Wien, Akad. Mat. Natur. 98 (1889) 473--484. 

\bibitem{Temam} Temam, R. \emph{Navier-Stokes equations. Theory and numerical analysis}. Revised edition. Studies in Mathematics and its Applications, 2. North-Holland Publishing Co., Amsterdam-New York, 1979. 

\bibitem{Max} Trokhimtchouk, M. \emph{Everywhere regularity of certain nonlinear diffusion systems}. Calc. Var. Partial Differential Equations 37 (2010), no. 3-4, 407--422. 

\bibitem{Tru} Trudinger, N. \emph{Pointwise estimates and quasilinear parabolic equations.} Comm. Pure Appl. Math. 21 1968 205--226. 

\bibitem{Vespri} Vespri, V. \emph{On the local behaviour of solutions of a certain class of doubly nonlinear parabolic equations.} Manuscripta Math. 75 (1992), no. 1, 65--80. 

\bibitem{Vis}  Visintin, A. \emph{Models of phase transitions}. Progress in Nonlinear Differential Equations and their Applications, 28. Birkh\"auser Boston, Inc., Boston, MA, 1996.

\bibitem{Woodruff} Woodruff, P. \emph{The Solid-Liquid Interface}. Cambridge University Press, Cambridge 1973.



\end{thebibliography}
\end{document}